\documentclass{amsart}
\usepackage{amsfonts, amsbsy, amsmath, amssymb}

\usepackage{pifont}

\makeatletter
\@namedef{subjclassname@2010}{%
  \textup{2020} Mathematics Subject Classification}
\makeatother

\usepackage{tikz}

\let\latexput\put
\usepackage{pictex}

\let\put\latexput

\ifx\fiverm\undefined 
  \newfont\fiverm{cmr5} 
\fi

\newtheorem{thm}{Theorem}[section]
\newtheorem{lem}[thm]{Lemma}

\newtheorem{thm-con}[thm]{Theorem-Conjecture}
\numberwithin{equation}{section}

\theoremstyle{definition}

\allowdisplaybreaks

\newcommand{\f}{\Bbb F}

\begin{document}

\title[Number of Equivalence Classes of Rational Functions]{Number of Equivalence Classes of Rational Functions over Finite Fields}

\author[Xiang-dong Hou]{Xiang-dong Hou}
\address{Department of Mathematics and Statistics,
University of South Florida, Tampa, FL 33620}
\email{xhou@usf.edu}



\begin{abstract}
Two rational functions $f,g\in\Bbb F_q(X)$ are said to be {\em equivalent} if there exist $\phi,\psi\in\Bbb F_q(X)$ of degree one such that $g=\phi\circ f\circ\psi$. We give an explicit formula for the number of equivalence classes of rational functions of a given degree in $\Bbb F_q(X)$. This result should provide guidance for the current and future work on classifications of low degree rational functions over finite fields. We also determine the number of equivalence classes of polynomials of a given degree in $\Bbb F_q[X]$.
\end{abstract}

\maketitle

\section{Introduction}

For a nonconstant rational function $f(X)$ over a field $\f$, written in the form $f(X)=P(X)/Q(X)$, where $P,Q\in\f[X]$, $Q\ne 0$, and $\text{gcd}(P,Q)=1$, we define $\deg f=\max\{\deg P,\deg Q\}$. Then $[\f(X):\f(f)]=\deg f$. By L\"uroth theorem, every subfield $E\subset\f(X)$ with $[\f(X):E]=d$ is of the form $\f(f)$ for some $f\in\f(X)$ with $\deg f=d$. Let 
\begin{equation}\label{G(F)}
G(\f)=\{\phi\in\f(X):\deg \phi=1\}.
\end{equation}
The group $(G(\f),\circ)$ is isomorphic to the projective linear group $\text{PGL}(2,\f)$ and the Galois group $\text{Aut}(\f(X)/\f)$ of $\f(X)$ over $\f$. For $A=\left[\begin{smallmatrix}a&b\cr c&d\end{smallmatrix}\right]\in\text{PGL}(2,\f)$, its corresponding element in $G(\f)$, denoted by $\phi_A$, is $\phi_A=(aX+b)/(cX+d)$. For $\phi\in G(\f)$, its corresponding element in $\text{Aut}(\f(X)/\f)$, denoted by $\sigma_\phi$, is the $\f$-automorphism of $\f(X)$ defined by $\sigma_\phi(X)=\phi(X)$.

Two rational functions $f,g\in\f(X)\setminus\f$ are said to be {\em equivalent}, denoted as $f\sim g$, if there exist $\phi,\psi\in G(\f)$ such that $g=\phi\circ f\circ \psi$. This happens if and only if $\f(g)=\sigma(\f(f))$ for some $\sigma\in\text{Aut}(\f(X)/\f)$. 

The set $\f(X)\setminus \f$ equipped with composition $\circ$ is a monoid and $G(\f)$ is the group of units of $(\f(X)\setminus\f,\circ)$. In a parallel setting, one replaces $\f(X)$ with $\f[X]$ and $G(\f)$ with the affine linear group $\text{AGL}(1,\f)=\{\phi\in\f[X]:\deg\phi=1\}$. Then $(\f[X]\setminus\f,\circ)$ is a submonoid of $(\f(X)\setminus\f,\circ)$ and $\text{AGL}(1,\f)$ is its group of units. If two polynomials $f,g\in\f[X]\setminus\f$ are equivalent as rational functions, i.e., $g=\phi\circ f\circ \psi$ for some $\phi,\psi\in G(\f)$, then there are $\alpha,\beta\in\text{AGL}(1,\f)$ such that $g=\alpha\circ f\circ\beta$; see Lemma~\ref{L8.1}. Factorizations in the monoids $(\f(X)\setminus\f,\circ)$ and $(\f[X]\setminus\f,\circ)$ are difficult questions that have attracted much attention \cite{Alonso-Gutierrez-Recio-JSC-1995, Ayad-Fleischmann-JSC-2008, Barton-Zippel-JSC-1985, Fuchs-Petho-PAMS-2011, Fuchs-Zannier-JEMS-2012, vonzurGathen-Weiss-JSC-1995, Zippel-ISSAC-1991}. Factorizations in $(\f(X)\setminus\f,\circ)$ are determined by the lattice $\mathcal L(\f)$ of the subfields of $\f(X)$ and vice versa. The Galois group $\text{Aut}(\f(X)/\f)$ is an automorphism group of $\mathcal L(\f)$ and the $\text{Aut}(\f(X)/\f)$-orbits in $\mathcal L(\f)$ correspond to the equivalence classes in $\f(X)\setminus\f$.

Many intrinsic properties of rational functions are preserved under equivalence. The degree of a rational function in $\f(X)\setminus\f$ is invariant under equivalence. If $f_1,f_2\in\f(X)\setminus\f$ are such that $f_2=\phi\circ f_1\circ\psi$ for some $\phi,\psi\in G(\f)$, then it is clear that $\sigma_\psi\in\text{Aut}(\f(X)/\f)$ resticts to an isomorphism $\tau:\f(f_1)\to\f(f_2)$ and extends to an isomorphism $\theta:N_1\to N_2$, where $N_i$ is the normal closure of $\f(X)$ over $\f(f_i)$.
\[
\begin{tikzpicture}[xscale=1, yscale=0.8]

\node at (0,0) {$\f$};
\node at (2,0) {$\f$};
\node at (0,2) {$\f(f_1)$};
\node at (2,2) {$\f(f_2)$};
\node at (0,4) {$\f(X)$};
\node at (2,4) {$\f(X)$};
\node at (0,6) {$N_1$};
\node at (2,6) {$N_2$};

\draw (0,0.5) -- (0,1.5);
\draw (2,0.5) -- (2,1.5);
\draw (0,2.5) -- (0,3.5);
\draw (2,2.5) -- (2,3.5);
\draw (0,4.5) -- (0,5.5);
\draw (2,4.5) -- (2,5.5);

\draw[->] (0.6,0) to (1.4,0);
\draw[->] (0.6,2) to (1.4,2);
\draw[->] (0.6,4) to (1.4,4);
\draw[->] (0.6,6) to (1.4,6);

\node [above] at (1,0) {$\scriptstyle \text{id}$};
\node [above] at (1,2) {$\scriptstyle \tau$};
\node [above] at (1,4) {$\scriptstyle \sigma_\psi$};
\node [above] at (1,6) {$\scriptstyle \theta$};

\end{tikzpicture}
\]
More interesting invariants of rational functions arise from this setting. For example, the number of ramification points of $f_1$ with a given ramification index, i.e., the number of degree one places of the function field $\f(X)$ with a given ramification index over $\f(f_1)$, equals that of $f_2$. (For the definition of the ramification index of a place in an extension of a function field, see \cite[III.1]{Stichtenoth-1993}.) The Galois group $\text{Aut}(N_1/\f(f_1))$, which is the arithmetic monodromy group of $f_1$ (\cite{Cohen-Matthews-TAMS-1994}), is isomorphic to that of $f_2$.
It is known that the equivalence classes of rational functions $f\in\f_q(X)\setminus\f_q$ such that $\f_q(X)/\f_q(f)$ is Galois are in one-to-one correspondence with the classes of conjugate subgroups of $\text{PGL}(2,\f_q)$; see \cite{Hou-CA-2020}.

When $\f=\f_q$, the finite field with $q$ elements, there is another important invariant: $|f(\Bbb P^1(\f_q))|$, the number of values of $f\in\f_q(X)$ on the projective line $\Bbb P^1(\f_q)$. In the theory and applications of finite fields, an important question is to understand the polynomials that permute $\f_q$ and the rational functions that permute $\Bbb P^1(\f_q)$ under the aforementioned equivalence. For classifications of low degree permutation polynomials of finite fields, see \cite{Dickson-AM-1896, Fan-arXiv1812.02080, Fan-arXiv1903.10309, Li-Chandler-Xiang-FFA-2010, Shallue-Wanless-FFA-2013}. Permutation rational functions of $\Bbb P^1(\f_q)$ of degree 3 and 4 were classified recently \cite{Ding-Zieve-arXiv2010.15657, Ferraguti-Micheli-DCC-2020, Hou-CA-2021}.

Equivalence of rational functions over finite fields also arises in other circumstances. Stichtenoth and Topuzo\u glu \cite{Stichtenoth-Topuzoglu-FFA-2012} introduced a natural action of $\text{PGL}(2,\f_q)$ on $\f_q[X]$ and studied the irreducible polynomials that are fixed by an element $A=\left[\begin{smallmatrix}a&b\cr c&d\end{smallmatrix}\right]\in\text{PGL}(2,\f_q)$; such polynomials are precisely the irreducible factors of $bX^{q^r+1}-aX^{q^r}+dX-c$ for some $r\ge 0$. Reis \cite{Reis-JPAA-2020} showed that the $A$-inavriant irreducible polynomials  are of the form $h^{\deg f}f(g/h)$, where $f\in\f_q[X]$ is irreducible and $g/h\in\f_q(X)$ is a rational function of degree $o(A)$ (the order of $A$ in $\text{PGL}(2,\f_q)$) in the lowest term. Moreover, the rational function $g/h$ is explicitly determined by $A$. Such $A$-invariant irreducible polynomials were enumerated in \cite{Reis-FFA-2020}. There are numerous papers on irreducible polynomials of the form $h^{\deg f}f(g/h)$ \cite{Ahmadi-FFA-2011, Mattarei-Pizzato-FFA-2017, Mattarei-Pizzato-FFA-2022, Panario-Reis-Wang-JPAA-2020}. Mattarei and Pizzato \cite{Mattarei-Pizzato-FFA-2017} showed that the number of irreducible polynomials of a given degree of the form $h^{\deg f}f(g/h)$ depends only on the equivalence class of the rational function $g/h$ .

We remind the reader that there is a fundamental difference between the PGL action defined in \cite{Stichtenoth-Topuzoglu-FFA-2012} and the PGL action considered in the present paper; the former acts on polynomials while the latter acts on rational functions. As such, our approach to the enumeration of the fixed points of the PGL action is substantially different from that in \cite{Reis-FFA-2020}.

When $\f=\f_q$, there are only finitely many equivalence classes of rational functions in $\f_q(X)\setminus\f_q$ with a given degree $n$. We shall denote this number by $\frak N(q,n)$. Despite its obvious significance, this number was not known previously.  The main contribution of the present paper is the determination of $\frak N(q,n)$ for all $q$ and $n$ (Theorem~\ref{T-main}). For example, when $n=3$, we have
\[
\frak N(q,3)=
\begin{cases}
2(q+1)&\text{if}\ q\equiv 1,4\pmod 6,\cr
2q&\text{if}\ q\equiv 2,5\pmod 6,\cr
2q+1&\text{if}\ q\equiv 3\pmod 6.
\end{cases}
\]

The classification of rational functions of degree $n\le 2$ over $\f_q$ is straightforward; see Sections~7.1 and 7.2. When $n=3$ and $q$ is even, the classification was obtained recently by Mattarei and Pizzato \cite{Mattarei-Pizzato-arXiv2104.00111} using the fact that such rational functions have at most two ramification points. The case $n=3$ and $q$ odd is still unsolved. (In this case, it was shown in \cite{Mattarei-Pizzato-arXiv2104.00111} that $\frak N(q,3)\le 4q$.)
A complete classification of rational functions over $\f_q$ appears to be out of reach. However, the determination of $\frak N(q,n)$ is an important step towards understanding the equivalence classes of rational functions over finite fields, especially those with low degree.

Here is the outline of our approach. There is an action of $\text{GL}(2,\f_q)$ on the set of subfields $F\subset\f_q(X)$ with $[\f_q(X):F]=n$, and $\frak N(q,n)$ is the number of orbits of this action. To compute $\frak N(q,n)$ by Burnside's lemma, it suffices to determine the number of such subfields of $\f_q(X)$ fixed by each member $A$ of $\text{GL}(2,\f_q)$. From there on, the computation becomes quite technical and depends on the canonical form of $A$.

The paper is organized as follows: In Section~2, we include some preliminaries and lay out the plan for computing $\frak N(q,n)$. In Sections 3 -- 5, we compute the number of subfields of $\f_q(X)$ of degree $n$ fixed by $A\in\text{GL}(2,\f_q)$ according to the canonical form of $A$. The case where the elementary divisor of $A$ is an irreducible quadratic, covered in Section~4, is significantly more difficult than the other cases. The explicit formula for $\frak N(q,n)$ is presented in Section~6. A discussion of low degree rational functions over $\f_q$ ensued in Section~7. The last section is devoted to equivalence classes of polynomials over finite fields. The situation is much simpler compared with the case of rational functions. The number of equivalence classes are computed and, as concrete examples, polynomials of degree up to 5 are classified. Several counting lemmas used in the paper are gathered in the appendix.

\section{Preliminaries}

\subsection{Rational functions and subfields}\

Let 
\begin{equation}\label{Rn}
\mathcal R_{q,n}=\{f\in\f_q(X):\deg f=n\}.
\end{equation}
By Lemma~A\ref{LA2},
\[
|\mathcal R_{q,n}|=\begin{cases}
q-1&\text{if}\ n=0,\cr
q^{2n-1}(q^2-1)&\text{if}\ n>0.
\end{cases}
\]
For $f_1,f_2\in\f_q(X)\setminus\f_q$, we define $f_1\sim f_2$ if $f_2=\phi\circ f_1\circ \psi$ for some $\phi,\psi\in G(\f_q)$ and we define $f_1\overset{\rm L}\sim f_2$ if there exists $\phi\in G(\f_q)$ such that $f_2=\phi\circ f_1$. It is clear that 
\[
f_1\overset{\rm L}\sim f_2 \ \Leftrightarrow\ \f_q(f_1)=\f_q(f_2)
\]
and
\[
f_1\sim f_2\ \Leftrightarrow\ \f_q(f_2)=\sigma(\f_q(f_1))\ \text{\rm for some}\ \sigma\in\text{Aut}(\f_q(X)/\f_q).
\]
Recall that $\frak N(q,n)$ denotes the number of $\sim$ equivalence classes in $\mathcal R_{q,n}$; this number is the main subject of our investigation.

For $f=P/Q\in\f_q(X)\setminus\f_q$, where $P,Q\in\f_q[X]$, $\text{gcd}(P,Q)=1$, let 
\[
\mathcal S(f)=\langle P,Q\rangle_{\f_q}=\{aP+bQ:a,b\in\f_q\},
\]
the $\f_q$-span of $\{P,Q\}$. (Throughout this paper, an $\f_q$-span is denoted by $\langle\ \rangle_{\f_q}$.) Then $f_1\overset{\rm L}\sim f_2$ $\Leftrightarrow$ $\mathcal S(f_1)=\mathcal S(f_2)$. By L\"uroth theorem, every subfield $F\subset\f_q(X)$ with $[\f_q(X):F]=n<\infty$ is of the form $F=\f_q(f)$, where $f\in\f_q(X)$ is of degree $n$. The number of such $F$ is
\[
\frac{|\mathcal R_{q,n}|}{|G(\f_q)|}=\frac{q^{2n-1}(q^2-1)}{q(q^2-1)}=q^{2(n-1)}.
\]
Denote the set of these fields by $\mathcal F_n=\{F_1,\dots,F_{q^{2(n-1)}}\}$ (Figure~\ref{F1}) and let $\text{Aut}(\f_q(X)/\f_q)$ act on $\mathcal F_n$. Then $\frak N(q,n)$ is precisely the number of orbits of this action.

\begin{figure}
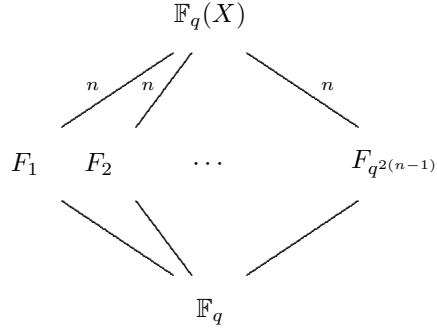

\[
\beginpicture
\setcoordinatesystem units <1.4em,1.4em> point at 0 0

\plot -1 -3  -4 -1 /
\plot 1 -3  4 -1 /
\plot -1 3  -4 1 /
\plot 1 3  4 1 /
\plot -0.5 -3  -2 -1 /
\plot -0.5 3  -2 1 /

\put {$\f_q$}  at 0 -4
\put {$\f_q(X)$}  at 0 4
\put {$F_1$}  at -5 0
\put {$F_2$}  at -3 0
\put {$F_{q^{2(n-1)}}$}  at 5 0
\put {$\cdots$} at 0 0
\put {$\scriptstyle n$} [br] at -3 2
\put {$\scriptstyle n$} [br] at -1.5 2
\put {$\scriptstyle n$} [bl] at 3 2
\endpicture
\]
\caption{Subfields of $\f_q(X)$ of degree $n$}\label{F1}
\end{figure}


\subsection{Conjugacy classes of $\text{GL}\boldsymbol{(2,\f_q)}$}\

Let 
\[
A_a=\left[\begin{matrix}a&0\cr 0&a\end{matrix}\right],\quad a\in\f_q^*,
\]
\[
A_{\{a,b\}}=\left[\begin{matrix}a&0\cr 0&b\end{matrix}\right],\quad a,b\in\f_q^*,
\]
\[
A_{\{\alpha,\alpha^q\}}=\left[\begin{matrix}\alpha+\alpha^q&-\alpha^{1+q}\cr 1&0\end{matrix}\right],\quad \alpha\in\f_{q^2}\setminus\f_q,
\]
\[
B_a=\left[\begin{matrix}a&a\cr 0&a\end{matrix}\right],\quad a\in\f_q^*.
\]
Then
\begin{align}\label{conj-cls}
\mathcal C:=\,&\{A_a:a\in\f_q^*\}\cup\{A_{\{a,b\}}:a,b\in\f_q^*,\ a\ne b\}\\
&\cup\{A_{\{\alpha,\alpha^q\}}:\alpha\in\f_q^2\setminus\f_q\}\cup\{B_a:a\in\f_q^*\}\nonumber
\end{align}
forms a set of representatives of the conjugacy classes of $\text{GL}(2,\f_q)$. Additional information about these representatives is given in Table~\ref{Tb1}, where $\text{cent}(A)$ denotes the centralizer of $A$ in $\text{GL}(2,\f_q)$; see \cite[\S 6.3]{Hou-ams-gsm-2018}.

\begin{table}[ht]
\caption{Conjugacy classes of $\text{GL}(2,\f_q)$}\label{Tb1}
   \renewcommand*{\arraystretch}{1.2}
    \centering
     \begin{tabular}{|c|c|c|}
         \hline
         $A\in\mathcal C$ & elementary divisors & $|\text{cent}(A)|$\\ \hline
         $A_a,\ a\in\f_q^*$ & $X-a,\ X-a$ &$q(q-1)^2(q+1)$ \\ \hline
         $A_{\{a,b\}},\ a,b\in\f_q^*,\ a\ne b$ & $X-a,\ X-b$ & $(q-1)^2$ \\ \hline
         $A_{\{\alpha,\alpha^q\}},\ \alpha\in\f_q^2\setminus\f_q$ & $(X-\alpha)(X-\alpha^q)$ & $q^2-1$ \\ \hline
         $B_a,\ a\in\f_q^*$ & $(X-a)^2$ & $q(q-1)$ \\ \hline
    \end{tabular}
\end{table}

\subsection{Burnside's lemma}\

Let $\text{GL}(n,\f_q)$ act on $\mathcal F_n$ as follows: For $A=\left[\begin{smallmatrix}a&b\cr c&d\end{smallmatrix}\right]\in\text{GL}(n,\f_q)$ and $\f_q(f)\in\mathcal F_n$, where $f\in\f_q(X)$ is of degree $n$, $A(\f_q(f))=\f_q(f\circ\phi_A)$, where $\phi_A=(aX+b)/(cX+d)$.
By Burnside's lemma,
\begin{align}\label{B-lem}
\frak N(q,n)=\,&\sum_{A\in\mathcal C}\frac{\text{Fix}(A)}{|\text{cent}(A)|}\\
=\,&\frac 1{q(q-1)^2(q+1)}\sum_{a\in\f_q^*}\text{Fix}(A_a)+\frac 1{(q-1)^2}\sum_{\{a,b\}\subset\f_q^*,\,a\ne b}\text{Fix}(A_{\{a,b\}})\cr
&+\frac 1{q^2-1}\sum_{\{\alpha,\alpha^q\}\subset\f_{q^2}\setminus\f_q}\text{Fix}(A_{\{\alpha,\alpha^q\}})+\frac 1{q(q-1)}\sum_{a\in\f_q^*}\text{Fix}(B_a),\nonumber
\end{align}
where
\[
\text{Fix}(A)=|\{F\in\mathcal F_n:A(F)=F\}|.
\]
Obviously, 
\begin{equation}\label{fixAa}
\text{Fix}(A_a)=|\mathcal F_n|=q^{2(n-1)}.
\end{equation}
We will determine $\text{Fix}(A_{\{a,b\}})$, $\text{Fix}(A_{\{\alpha,\alpha^q\}})$, and $\text{Fix}(B_a)$ in the subsequent sections; in doing so, we will need a number of counting lemmas which are given in the appendix.


\section{Determination of $\text{Fix}(A_{\{a,b\}})$}

Let $a,b\in\f_q^*$, $a\ne b$ and $c=a/b$. Then $\phi_{A_{\{a,b\}}}=cX$. Therefore, a field $\f_q(f)$, where $f\in\f_q(X)\setminus\f_q$, is fixed by $A_{\{a,b\}}$ if and only if $\f_q(f(X))=\f_q(f(cX))$.

\begin{lem}\label{L2.0} Let $f\in\f_q(X)$ with $\deg f=n>0$ and $1\ne c\in\f_q^*$ with $o(c)=d$, where $o(c)$ denotes the multiplicative order of $c$. Then $\f_q(f(X))=\f_q(f(cX))$ if and only if
\[
\mathcal S(f)=\langle X^{r_1}P_1(X^d),\,X^{r_2}Q_1(X^d)\rangle_{\f_q},
\]
where $P_1,Q_1\in\f_q[X]$ are monic, $0\le r_1,r_2<d$, $\deg(X^{r_2}Q_1(X^d))<\deg(X^{r_1}P_1(X^d))=n$, and $\text{\rm gcd}(X^{r_1}P_1,\,X^{r_2}Q_1)=1$.
\end{lem}

\begin{proof}
($\Leftarrow$) Obvious.

\medskip
($\Rightarrow$) We may assume that $f=P/Q$, where $P,Q\in\f_q[X]$ are monic, $\deg P=n$, $\deg Q=m<n$, $\text{gcd}(P,Q)=1$, and the coefficient of $X^m$ in $P$ is 0. Let $n\equiv r_1\pmod d$ and $m\equiv r_2\pmod d$, where $0\le r_1,r_2<d$. Such a pair $(P,Q)$ is uniquely determined by $\mathcal S(f)$. Since
\[
\langle P(X),Q(X)\rangle_{\f_q}=\mathcal S(f)=\mathcal S(f(cX))=\langle c^{-n}P(cX),c^{-m}Q(cX)\rangle_{\f_q},
\]
we have
\[
P(X)=c^{-n}P(cX),\quad Q(X)=c^{-m}Q(cX).
\]
Thus the coefficient of $X^i$ in $P(X)$ is 0 for all $i$ with $i\not\equiv n\pmod d$, whence $P(X)=x^{r_1}P_1(X^d)$. In the same way, $Q(X)=X^{r_2}Q_1(X^d)$. Since $\text{gcd}(P,Q)=1$, we have $\text{gcd}(X^{r_1}P_1,X^{r_2}Q_1)=1$.
\end{proof}

In Lemma~\ref{L2.0}, let $m=\deg(X^{r_2}Q_1(X^d))$. Note that $\text{\rm gcd}(X^{r_1}P_1,\,X^{r_2}Q_1)=1$ if and only if $\text{gcd}(P_1,Q_1)=1$ plus one of the following: (i) $r_1=r_2=0$; (ii) $r_1=0$, $r_2>0$, $P_1(0)\ne 0$; (iii) $r_1>0$, $r_2=0$, $Q_1(0)\ne 0$. When $r_1=r_2=0$, i.e., $n\equiv m\equiv 0\pmod d$, the number of the fields $\f_q(f)$ in Lemma~\ref{L2.0} fixed by $A_{\{a,b\}}$ is $q^{-1}\alpha_{m/d,n/d}$, where
\[
\alpha_{i,j}=|\{(f,g):f,g\in\f_q[X]\ \text{monic},\ \deg f=i,\ \deg g=j,\ \text{gcd}(f,g)=1\}|.
\]
When $r_1=0$ and $r_2>0$, i.e., $n\equiv 0\pmod d$ but $m\not\equiv 0\pmod d$, the number of $\f_q(f)$ fixed by $A_{\{a,b\}}$ is $\beta_{n/d,\lfloor m/d\rfloor}$, where
\[
\beta_{i,j}=|\{(f,g):f,g\in\f_q[X]\ \text{monic},\ \deg f=i,\ \deg g=j,\ f(0)\ne0,\ \text{gcd}(f,g)=1\}|.
\]
When $r_1>0$ and $r_2=0$, i.e., $m\equiv 0\pmod d$ but $n\not\equiv 0\pmod d$, the number of $\f_q(f)$ fixed by $A_{\{a,b\}}$ is $\beta_{m/d,\lfloor n/d\rfloor}$.

Define
\[
\alpha_j=|\{(f,g):f,g\in\f_q[X]\ \text{monic},\ \deg f<j,\ \deg g=j,\ \text{gcd}(f,g)=1\}|=\sum_{0\le i\le j}\alpha_{i,j}.
\]
The numbers $\alpha_{i,j}$, $\alpha_j$ and $\beta_{i,j}$ are determined in Appendix, Lemmas~A\ref{LA1} and A\ref{LA3}.

\begin{thm}\label{T2.1}
Let $a,b\in\f_q^*$, $a\ne b$, and $d=o(a/b)$. Then
\[
\text{\rm Fix}(A_{\{a,b\}})=\begin{cases}
\displaystyle q^{2n/d-2}+\frac{(d-1)(q^{2n/d}-1)}{q+1}&\text{if}\ n\equiv 0\pmod d,\vspace{0.5em}\cr
\displaystyle \frac{q^{2\lfloor n/d\rfloor+1}+1}{q+1}&\text{if}\ n\not\equiv 0\pmod d.
\end{cases}
\]
\end{thm}

\begin{proof}
If $n\equiv 0\pmod d$, using Lemmas~A\ref{LA1} and A\ref{LA3}, we have
\begin{align*}
\text{\rm Fix}(A_{\{a,b\}})\,&=\sum_{\substack{0\le m<n\cr m\equiv0\,(\text{mod}\,d)}}q^{-1}\alpha_{m/d,n/d}+\sum_{\substack{0\le m<n\cr m\not\equiv0\,(\text{mod}\,d)}}\beta_{n/d,\lfloor m/d\rfloor}\cr
&=q^{-1}\sum_{0\le i<n/d}\alpha_{i,n/d}+\sum_{0\le i<n/d}(d-1)\beta_{n/d,i}\cr
&=q^{-1}\alpha_{n/d}+(d-1)\sum_{0\le i<n/d}q^{n/d-i-1}(q-1)\frac{q^{2i+1}+1}{q+1}\cr
&=q^{2n/d-2}+\frac{(d-1)(q-1)}{q+1}\sum_{0\le i<n/d}(q^{n/d}\cdot q^i+q^{n/d-1-i})\cr
&=q^{2n/d-2}+\frac{(d-1)(q-1)}{q+1}\Bigl( q^{n/d}\frac{q^{n/d}-1}{q-1}+\frac{q^{n/d}-1}{q-1}\Bigr)\cr
&=q^{2n/d-2}+\frac{(d-1)(q^{2n/d}-1)}{q+1}.
\end{align*}
If $n\not\equiv 0\pmod d$, we have
\begin{align*}
\text{\rm Fix}(A_{\{a,b\}})\,&=\sum_{\substack{0\le m<n\cr m\equiv0\,(\text{mod}\,d)}}\beta_{m/d,\lfloor n/d\rfloor}=\sum_{0\le i\le\lfloor n/d\rfloor}\beta_{i,\lfloor n/d\rfloor}\cr
&=q^{\lfloor n/d\rfloor}+\sum_{1\le i\le\lfloor n/d\rfloor}q^{\lfloor n/d\rfloor-i}(q-1)\frac{q^{2i}-1}{q+1}\kern2em\text{(by Lemma~A\ref{LA3})}\cr
&=q^{\lfloor n/d\rfloor}+\frac{q-1}{q+1}\sum_{1\le i\le\lfloor n/d\rfloor}(q^{\lfloor n/d\rfloor+1}\cdot q^{i-1}-q^{\lfloor n/d\rfloor-i})\cr
&=q^{\lfloor n/d\rfloor}+\frac{q-1}{q+1}\Bigl(q^{\lfloor n/d\rfloor+1}\,\frac{q^{\lfloor n/d\rfloor}-1}{q-1}-\frac{q^{\lfloor n/d\rfloor}-1}{q-1}\Bigr)\cr
&=q^{\lfloor n/d\rfloor}+\frac{(q^{\lfloor n/d\rfloor}-1)(q^{\lfloor n/d\rfloor+1}-1)}{q+1}\cr
&=\frac{q^{2\lfloor n/d\rfloor+1}+1}{q+1}.
\end{align*}
\end{proof}


\section{Determination of $\text{Fix}(A_{\{\alpha,\alpha^q\}})$}

Let 
\[
A=A_{\{\alpha,\alpha^q\}}=\left[\begin{matrix}\alpha+\alpha^q&-\alpha^{1+q}\cr 1&0\end{matrix}\right],\quad \alpha\in\f_{q^2}\setminus\f_q.
\]
We have
\begin{equation}\label{2.1}
BAB^{-1}=D,
\end{equation}
where 
\[
D=\left[\begin{matrix}\alpha^q&0\cr 0&\alpha\end{matrix}\right],\quad B=\left[\begin{matrix}1&-\alpha\cr 1&-\alpha^q\end{matrix}\right]\in\text{GL}(2,\f_{q^2}).
\]
Note that $\phi_D=\alpha^{q-1}X\in G(\f_{q^2})$.

\begin{lem}\label{L2.2}
Let $f\in\f_q(X)\setminus\f_q$ and $g=f\circ\phi_B^{-1}\in\f_{q^2}(X)$. Then $\f_q(f)$ is fixed by $A$ if and only if $\f_{q^2}(g)$ is fixed by $D$.
\end{lem}

\begin{proof}
We have 
\begin{align*}
& \f_{q^2}(g)\ \text{is fixed by}\ D\cr
\Leftrightarrow\ &g\circ\phi_D=\psi\circ g\ \text{for some}\ \psi\in G(\f_{q^2})\cr
\Leftrightarrow\ &f\circ\phi_A=\psi\circ f\ \text{for some}\ \psi\in G(\f_{q^2})&\text{(by \eqref{2.1})}\cr
\Leftrightarrow\ &f\circ\phi_A=\psi\circ f\ \text{for some}\ \psi\in G(\f_{q})&\text{(by Lemma~\ref{L2.3})}\cr
\Leftrightarrow\ &\f_q(f)\ \text{is fixed by}\ A.&
\end{align*}
\end{proof}

\begin{lem}\label{L2.3}
Let $f_1,f_2\in\f_q(X)\setminus\f_q$ be such that there exists $\psi\in G(\f)$, where $\f$ is an extension of $\f_q$, such that $f_2=\psi\circ f_1$. Then there exists $\theta\in G(\f_q)$ such that $f_1=\theta\circ f_2$.
\end{lem}

\begin{proof}
Let $f_i=P_i/Q_i$, where $P_i,Q_i\in\f_q[X]$ and $\text{gcd}(P_i,Q_i)=1$. It suffices to show that there exist $a_0,b_0,c_0,d_0\in\f_q$ such that 
\begin{equation}\label{2.12}
\left[\begin{matrix}a_0&b_0\cr c_0&d_0\end{matrix}\right]\left[\begin{matrix}P_1\cr Q_1\end{matrix}\right]=\left[\begin{matrix}P_2\cr Q_2\end{matrix}\right].\end{equation}
By assumption, there exist $a,b,c,d\in\f$ such that 
\[
\left[\begin{matrix}a&b\cr c&d\end{matrix}\right]\left[\begin{matrix}P_1\cr Q_1\end{matrix}\right]=\left[\begin{matrix}P_2\cr Q_2\end{matrix}\right].
\]
Write $\f=\f_q\oplus V$ as a direct sum of $\f_q$-subspaces, and write $a=a_0+a_1$, $b=b_0+b_1$, $c=c_0+c_1$, $d=d_0+d_1$, where $a_0,b_0,c_0,d_0\in\f_q$ and $a_1,b_1,c_1,d_1\in V$. Then
\[
\left[\begin{matrix}a_0&b_0\cr c_0&d_0\end{matrix}\right]\left[\begin{matrix}P_1\cr Q_1\end{matrix}\right]+\left[\begin{matrix}a_1&b_1\cr c_1&d_1\end{matrix}\right]\left[\begin{matrix}P_1\cr Q_1\end{matrix}\right]=\left[\begin{matrix}P_2\cr Q_2\end{matrix}\right].
\]
Comparing the coefficients in the above gives \eqref{2.12}.
\end{proof}

\begin{lem}\label{L2.4}
For $g\in\f_{q^2}(X)$, $g\circ\phi_B\in\f_q(X)$ if and only if $\bar g(X)=g(X^{-1})$, where $\bar g$ denotes the rational function obtained by applying $(\ )^q$ to the coefficients of $g$.
\end{lem}

\begin{proof}
Recall that $\phi_B(X)=(X-\alpha)/(X-\alpha^q)$. Since $\bar\phi_B=X^{-1}\circ\phi_B$, we have
\begin{align*}
g\circ\phi_B\in\f_q(X)\ \Leftrightarrow\ &\overline{g\circ\phi_B}=g\circ\phi_B\cr
\Leftrightarrow\ &\bar g\circ X^{-1}\circ\phi_B=g\circ\phi_B\cr
\Leftrightarrow\ &\bar g=g\circ X^{-1}.
\end{align*}
\end{proof}

Lemmas~\ref{L2.2} and \ref{L2.4} suggest the following strategy (which we will follow) to determine $\text{Fix}(A_{\{\alpha,\alpha^q\}})$:

\begin{itemize}
\item[Step 1.] Determine all $g\in\f_{q^2}(X)$ of degree $n$ such that $\f_{q^2}(g(\alpha^{q-1}X))=\f_{q^2}(g(X))$.

\item[Step 2.] Among all $g$'s in Step 1, determine those such that $\bar g(X)=g(X^{-1})$.

\item[Step 3.] Conclude that $\text{Fix}(A_{\{\alpha,\alpha^q\}})=|G(\f_q)|^{-1}\cdot(\text{the number of $g$'s in Step 2})$.
\end{itemize}    

We now carry out these steps in detail.

\medskip
{\bf Step 1.} Determine all $g\in\f_{q^2}(X)$ of degree $n$ such that $\f_{q^2}(g(\alpha^{q-1}X))=\f_{q^2}(g(X))$.

Let $d=o(\alpha^{q-1})$. By Lemma~\ref{L2.0}, for $g\in\f_{q^2}(X)$ with $\deg g=n$, $\f_{q^2}(g(\alpha^{q-1}X))=\f_{q^2}(g(X))$ if and only if 
\begin{equation}\label{2.13}
\mathcal S(g)=\langle X^{r_1}P_1(X^d),\,X^{r_2}Q_1(X^d)\rangle_{\f_{q^2}},
\end{equation}
where $0\le r_1,r_2<d$, $P_1,Q_1\in\f_{q^2}[X]$ are monic, $\deg(X^{r_2}Q_1(X^d))<\deg(X^{r_1}P_1(X^d))=n$, $\text{gcd}(X^{r_1}P_1,X^{r_2}Q_1)=1$, and $\langle\ \ \rangle_{\f_{q^2}}$ is the $\f_{q^2}$-span.

In \eqref{2.13}, let $m=\deg(X^{r_2}Q_1(X^d))$. Note that $n\equiv r_1\pmod d$, $m\equiv r_2\pmod d$, $\text{gcd}(P_1,Q_1)=1$, and one of the following holds: (i) $r_1=r_2=0$; (ii) $r_1=0$, $r_2>0$, $P_1(0)\ne 0$; (iii) $r_1>0$, $r_2=0$, $Q_1(0)\ne 0$. Let $g\in\f_{q^2}(X)$ satisfy \eqref{2.13}, i.e., 
\begin{equation}\label{2.14}
g=\frac{sX^{r_1}P_1(X^d)+tX^{r_2}Q_1(X^d)}{uX^{r_1}P_1(X^d)+vX^{r_2}Q_1(X^d)},
\end{equation}
where $\left[\begin{smallmatrix}s&t\cr u&v\end{smallmatrix}\right]\in\text{GL}(2,\f_{q^2})$.

\medskip
{\bf Step 2.} Among all $g$'s in Step 1, determine those such that $\bar g(X)=g(X^{-1})$.

For fixed $r_1$ and $r_2$, let
\[
N(r_1,r_2)=\text{the number of $g$ satisfying \eqref{2.13} and $\bar g(X)=g(X^{-1})$}.
\] 

\medskip
{\bf Case (i)} Assume $r_1=r_2=0$. In this case, we may write \eqref{2.14} as
\[
g(X)=\epsilon\frac{P(X^d)}{Q(X^d)},
\]
where $\epsilon\in\f_{q^2}^*$, $P,Q\in\f_{q^2}[X]$ are monic, $\deg P=l_1$, $\deg Q=l_2$, $\max\{l_1,l_2\}=n/d=:k$, and $\text{gcd}(P,Q)=1$. Then
\[
g(X^{-1})=\epsilon\frac{X^{dk}P(X^{-d})}{X^{dk}Q(X^{-d})},
\]
so $\bar g(X)=g(X^{-1})$ if and only if
\begin{align}\label{2.15}
\bar\epsilon\overline P(X)\,&=c\epsilon X^k P(X^{-1}),\\ \label{2.16}
\overline Q(X)\,&=cX^k Q(X^{-1})
\end{align}
for some $c\in\f_{q^2}^*$.

First, assume that $l_1=k$, and $l_2\le k$ is fixed. Then \eqref{2.16} is equivalent to
\[
Q(X)=X^{k-l_2}Q_1(X),
\]
where $Q_1\in\f_{q^2}[X]$, $\deg Q_1=2l_2-k$ (thus $k/2\le l_2\le k$), and
\begin{equation}\tag{\ref{2.16}$'$}\label{2.16'}
\overline Q_1(X)=cX^{2l_2-k}Q_1(X^{-1}).
\end{equation}
We call a polynomial $f\in\f_{q^2}[X]\setminus\{0\}$ {\em self-dual} if $X^{\deg f}\bar f(X^{-1})=cf(X)$ for some $c\in\f_{q^2}^*$. 
Thus, if $g$ satisfies \eqref{2.13} and $\bar g(X)=g(X^{-1})$, then both $P$ and $Q_1$ are self-dual. On the other hand, if both $P$ and $Q_1$ are self-dual, then the $c$ in \eqref{2.16} belongs to $\mu_{q+1}:=\{x\in\f_{q^2}:x^{q+1}=1\}$ and $c$ is uniquely determined by $Q_1$. Subsequently, in \eqref{2.15}, $\epsilon^{q-1}$ is uniquely determined and there are $q-1$ choices for $\epsilon$. Therefore, in this case, the number of $g$ satisfying \eqref{2.13} and $\bar g(X)=g(X^{-1})$ is 
\begin{align}\label{2.17}
&(q-1)\,|\{(P,Q_1): P,Q_1\in\f_{q^2}[X]\ \text{are monic and self-dual},\\
&\kern7.6em \deg P=k,\ \deg Q_1=2l_2-k,\ \text{gcd}(P,Q_1)=1\}|\cr
=\,&(q-1)\Gamma_{k,2l_2-k},
\nonumber
\end{align}
where
\begin{align*}
\Gamma_{i,j}=\{(f_1,f_2):\,& f_1,f_2\in\f_{q^2}[X]\ \text{are monic and self-dual},\\
&\deg f_1=i,\ \deg f_2=j,\ \text{gcd}(f_1,f_2)=1\}|.
\end{align*}
The number $\Gamma_{i,j}$ is determined in Appendix, Lemma~A\ref{LA5}.

Next, assume that $l_2=k$ and $l_1<k$ is fixed. By the same argument, the number of $g$ satisfying \eqref{2.13} and $\bar g(X)=g(X^{-1})$ is $(q-1)\Gamma_{k,2l_1-k}$.

Therefore, the total number of $g$ satisfying \eqref{2.13} and $\bar g(X)=g(X^{-1})$ in Case (i) is
\begin{align*}
N(0,0)\,&=(q-1)\sum_{k/2\le l_2\le k}\Gamma_{k,2l_2-k}+(q-1)\sum_{k/2\le l_1< k}\Gamma_{k,2l_1-k}\cr
&=(q-1)\Bigl(2\sum_{\substack{0\le i<k\cr i\equiv k\,(\text{mod}\,2)}}\Gamma_{i,k}+\Gamma_{k,k}\Bigr).
\end{align*}
If $k=2k_1$,
\begin{align*}
N(0,0)\,&=(q-1)\Bigl(2\sum_{0\le i<k_1}\Gamma_{2i,2k_1}+\Gamma_{2k_1,2k_1}\Bigr)\cr
&=(q-1)\Bigl[2\Bigl(q^{2k_1-1}(q+1)+\sum_{1\le i<k_1}\frac{q^{2(k_1-i)-1}(q+1)(q^2-1)}{q^2+1}(q^{4i}-1)\Bigr)\cr
&\kern 5em +\frac{q(q+1)}{q^2+1}(q^{4k_1}-q^{4k_1-2}-2)\Bigr]\kern5.3em\text{(by Lemma~A\ref{LA5})}\cr
&=(q-1)\Bigl[2q^{2k_1-1}(q+1)+2\frac{(q+1)(q^2-1)q^{2k_1-1}}{q^2+1}\sum_{1\le i<k_1}(q^{2i}-q^{-2i})\cr
&\kern 5em +\frac{q(q+1)}{q^2+1}(q^{4k_1}-q^{4k_1-2}-2)\Bigr]\cr
&=(q^2-1)\Bigl[2q^{2k_1-1}+\frac{2(q^2-1)q^{2k_1-1}}{q^2+1}\Bigl(q^2\frac{1-q^{2(k_1-1)}}{1-q^{2}}-q^{-2}\frac{1-q^{-2(k_1-1)}}{1-q^{-2}}\Bigr)\cr
&\kern 5em +\frac{q}{q^2+1}(q^{4k_1}-q^{4k_1-2}-2)\Bigr]\cr
&=(q^2-1)q^{4k_1-1}.
\end{align*}
If $k=2k_1+1$,
\begin{align*}
N(0,0)\,&=(q-1)\Bigl(2\sum_{0\le i<k_1}\Gamma_{2i+1,2k_1+1}+\Gamma_{2k_1+1,2k_1+1}\Bigr)\cr
&=(q-1)\Bigl[2\sum_{0\le i<k_1}\frac{q^{2(k_1-i)-1}(q+1)(q^2-1)}{q^2+1}(q^{4i+2}+1)\cr
&\kern 5em +\frac{q(q+1)}{q^2+1}(q^{4k_1+2}-q^{4k_1}+2)\Bigr] \kern5em\text{(by Lemma~A\ref{LA5})}\cr
&=(q-1)\Bigl[2\frac{(q+1)(q^2-1)q^{2k_1-1}}{q^2+1}\sum_{0\le i<k_1}(q^{2i+2}+q^{-2i})\cr
&\kern 5em +\frac{q(q+1)}{q^2+1}(q^{4k_1+2}-q^{4k_1}+2)\Bigr]\cr
&=(q^2-1)\Bigl[\frac{2(q^2-1)q^{2k_1-1}}{q^2+1}\Bigl( q^2\frac{1-q^{2k_1}}{1-q^2}+\frac{1-q^{-2k_1}}{1-q^{-2}}\Bigr)\cr
&\kern 5em +\frac{q}{q^2+1}(q^{4k_1+2}-q^{4k_1}+2)\Bigr]\cr
&=(q^2-1)q^{4k_1+1}.
\end{align*}
Therefore, we always have
\begin{equation}\label{case1}
N(0,0)=(q^2-1)q^{2k-1}.
\end{equation}

\medskip

{\bf Case (ii)} Assume $r_1=0$, $r_2>0$ and $P_1(0)\ne 0$. By \eqref{2.14},
\[
g(X^{-1})=\frac{sX^nP_1(X^{-d})+tX^{n-r_2}Q_1(X^{-d})}{uX^nP_1(X^{-d})+vX^{n-r_2}Q_1(X^{-d})}.
\]
Hence $\bar g(X)=g(X^{-1})$ if and only if 
\[
\begin{cases}
\bar s\overline P_1(X^d)+\bar tX^{r_2}\overline Q_1(X^d)=c\bigl[sX^nP_1(X^{-d})+tX^{n-r_2}Q_1(X^{-d})\bigr],\vspace{0.2em}\cr
\bar u\overline P_1(X^d)+\bar vX^{r_2}\overline Q_1(X^d)=c\bigl[uX^nP_1(X^{-d})+vX^{n-r_2}Q_1(X^{-d})\bigr]
\end{cases}
\]
for some $c\in\f_{q^2}^*$, which is equivalent to
\begin{equation}\label{2.18}
\begin{cases}
\bar s\overline P_1(X^d)=csX^nP_1(X^{-d}),\cr
\bar u\overline P_1(X^d)=cuX^nP_1(X^{-d}),\cr
\bar tX^{r_2}\overline Q_1(X^d)=ctX^{n-r_2}Q_1(X^{-d}),\cr
\bar vX^{r_2}\overline Q_1(X^d)=cvX^{n-r_2}Q_1(X^{-d}).
\end{cases}
\end{equation}
Let $k=n/d$ and $l=(m-r_2)/d$. The above equations imply that $\overline P_1(X)$ self-dual and $\overline Q_1(X^d)=\delta X^{n-2r_2}Q_1(X^{-d})$ for some $\delta\in\f_{q^2}^*$. It is necessary that $n-2r_2\equiv 0\pmod d$, i.e., $d$ is even and $r_2=d/2$. Hence $\overline Q_1(X)=\delta X^{k-1}Q_1(X^{-1})$. It follows that $Q_1(X)=X^{k-l-1}Q_2(X)$, where $Q_2(X)$ is monic and self-dual of degree $2l-k+1$. (So $(k-1)/2\le l\le k-1$.)

On the other hand, let $P_1,Q_2\in\f_{q^2}[X]$ be monic and self-dual with $\deg P_1=k$ and $\deg Q_2=2l-k+1$ ($(k-1)/2\le l\le k-1$). Then $\overline P_1(X)=\epsilon X^kP_1(X^{-1})$ and $\overline Q_2(X)=\delta X^{2l-k+1}Q_2(X^{-1})$ for some $\epsilon,\delta\in\mu_{q+1}$. 
Let $Q_1(X)=X^{k-l-1}Q_2(X)$. Then \eqref{2.18} is satisfied if and only if
\begin{equation}\label{2.19}
\begin{cases}
\bar s\epsilon=cs,\cr
\bar u\epsilon=cu,\cr
\bar t\delta=ct,\cr
\bar v\delta=cv.
\end{cases}
\end{equation}
Under the assumption that $\det\left[\begin{smallmatrix}s&t\cr u&v\end{smallmatrix}\right]\ne 0$, \eqref{2.19} implies that $c\in\mu_{q+1}$. Write $\epsilon=\epsilon_0^{q-1}$, $\delta=\delta_0^{q-1}$ and $c=c_0^{q-1}$, where $\epsilon_0,\delta_0,c_0\in\f_{q^2}^*$. Then \eqref{2.19} is satisfied if and only if
\[
\left[\begin{matrix}s&t\cr u&v\end{matrix}\right]=\left[\begin{matrix}s_1c_0/\epsilon_0&t_1c_0/\delta_0\cr u_1c_0/\epsilon_0&v_1c_0/\delta_0\end{matrix}\right],
\]
where $s_1,t_1,u_1,v_1\in\f_q$. Therefore, the number of $\left[\begin{smallmatrix}s&t\cr u&v\end{smallmatrix}\right]$ satisfying \eqref{2.18} is
\[
(q+1)\,|\text{GL}(2,\f_q)|=q(q^2-1)^2.
\]

To recap, when $d$ is even, $r_2=d/2$ and $l$ ($(k-1)/2\le l\le k-1$) is fixed, the number of $g$ satisfying \eqref{2.13} and $\bar g(X)=g(X^{-1})$ is
\[
\frac 1{q^2-1}q(q^2-1)^2\,\Gamma_{k,2l-k+1}=q(q^2-1)\Gamma_{k,2l-k+1}.
\]
Hence, when $d$ is even, 
\[
N(0,r_2)=q(q^2-1)\sum_{(k-1)/2\le l\le k-1}\Gamma_{k,2l-k+1}=q(q^2-1)\sum_{\substack{0\le i\le k-1\cr i\equiv k-1\,(\text{mod}\,2)}}\Gamma_{i,k}.
\]
In the above, if $k=2k_1$, 
\begin{align*}
N(0,r_2)\,&=q(q^2-1)\sum_{1\le i\le k_1}\Gamma_{2i-1,2k_1}\cr
&=q(q^2-1)\sum_{1\le i\le k_1}\frac{q^{2k_1-(2i-1)-1}(q+1)(q^2-1)}{q^2+1}(q^{4i-2}+1)\cr
&\kern18.5em \text{(by Lemma~A\ref{LA5})}\cr
&=\frac{q(q+1)(q^2-1)^2}{q^2+1}q^{2k_1}\sum_{1\le i\le k_1}(q^{2i-2}+q^{-2i})\cr
&=\frac{(q+1)(q^2-1)^2 q^{2k_1+1}}{q^2+1}\Bigl(\frac{1-q^{2k_1}}{1-q^2}+q^{-2}\frac{1-q^{-2k_1}}{1-q^{-2}}\Bigr)\cr
&=\frac{(q+1)(q^2-1)^2 q^{2k_1+1}}{q^2+1}\cdot\frac{q^{-2k_1}(q^{4k_1}-1)}{q^2-1}\cr
&=\frac{q(q+1)(q^2-1)(q^{4k_1}-1)}{q^2+1}.
\end{align*}
If $k=2k_1+1$,
\begin{align*}
N(0,r_2)\,&=q(q^2-1)\sum_{0\le i\le k_1}\Gamma_{2i,2k_1+1}\cr
&=q(q^2-1)\Bigl[q^{2k_1}(q+1)+\sum_{1\le i\le k_1}\frac{q^{2k_1+1-2i-1}(q+1)(q^2-1)}{q^2+1}(q^{4i}-1)\Bigr]\cr
&\kern23.3em \text{(by Lemma~A\ref{LA5})}\cr
&=q(q^2-1)(q+1)\Bigl[q^{2k_1}+\frac{(q^2-1)q^{2k_1}}{q^2+1}\sum_{1\le i\le k_1}(q^{2i}-q^{-2i})\Bigr]\cr
&=q(q^2-1)(q+1)\Bigl[q^{2k_1}+\frac{(q^2-1)q^{2k_1}}{q^2+1}\Bigl(q^2\frac{1-q^{2k_1}}{1-q^2}-q^{-2}\frac{1-q^{-2k_1}}{1-q^{-2}}\Bigr)\Bigr]\cr
&=q(q^2-1)(q+1)\frac{1+q^{4k_1+2}}{1+q^2}\cr
&=\frac{q(q+1)(q^2-1)(q^{4k_1+2}+1)}{q^2+1}.
\end{align*}

To summarize, we have
\begin{equation}\label{case2}
N(0,r_2)=\begin{cases}
\displaystyle \frac{q(q+1)(q^2-1)(q^{2k}-(-1)^k)}{q^2+1}&\text{if $d$ is even},\cr
0&\text{if $d$ is odd}.
\end{cases}
\end{equation}

\medskip
{\bf Case (iii)} Assume $r_1>0$, $r_2=0$ and $Q_1(0)\ne 0$. By \eqref{2.14},
\[
g(X^{-1})=\frac{sX^{n-r_1}P_1(X^{-d})+tX^nQ_1(X^{-d})}{uX^{n-r_1}P_1(X^{-d})+vX^nQ_1(X^{-d})}.
\]
Hence $\bar g(X)=g(X^{-1})$ if and only if
\[
\begin{cases}
\bar sX^{r_1}\overline{P_1}(X^{d})+\bar t\,\overline{Q_1}(X^{d})=c\bigl[sX^{n-r_1}P_1(X^{-d})+tX^nQ_1(X^{-d}) \bigr],\vspace{0.2em}\cr
\bar uX^{r_1}\overline{P_1}(X^{d})+\bar v\overline{Q_1}(X^{d})=c\bigl[uX^{n-r_1}P_1(X^{-d})+vX^nQ_1(X^{-d}) \bigr]
\end{cases}
\]
for some $c\in\f_{q^2}^*$, which is equivalent to
\begin{equation}\label{2.20}
\begin{cases}
\bar sX^{r_1}\overline{P_1}(X^d)=ctX^nQ_1(X^{-d}),\cr
\bar uX^{r_1}\overline{P_1}(X^d)=cvX^nQ_1(X^{-d}),\cr
\bar t\,\overline{Q_1}(X^d)=csX^{n-r_1}P_1(X^{-d}),\cr
\bar v\overline{Q_1}(X^d)=cuX^{n-r_1}P_1(X^{-d}).
\end{cases}
\end{equation}
Under the assumption that $\det\left[\begin{smallmatrix}s&t\cr u&v\end{smallmatrix}\right]\ne 0$, \eqref{2.20} implies that $s,t,u,v\ne 0$ and $c\in\mu_{q+1}$. Without loss of generality, assume $s=1$. Then \eqref{2.20} becomes
\begin{equation}\label{2.21}
\begin{cases}
\overline{P_1}(X)=ctX^kQ_1(X^{-1}),\cr
c\in\mu_{q+1},\cr
v=\bar ut,
\end{cases}
\end{equation}
where $k=(n-r_1)/d$. Moreover, 
\[
\det\left[\begin{matrix}1&t\cr u&v\end{matrix}\right]=\det\left[\begin{matrix}1&t\cr u&\bar ut\end{matrix}\right]=t(\bar u-u),
\]
which is nonzero if and only if $t\in\f_{q^2}^*$ and $u\in\f_{q^2}\setminus\f_q$.

Condition \eqref{2.21} implies that
\[
\widetilde P_1:=X^k\overline P_1(X^{-1})=ctQ_1(X),
\]
where $\text{gcd}(P_1,\widetilde P_1)=\text{gcd}(P_1,Q_1)=1$.

On the other hand, to satisfy \eqref{2.21} with $u\in\f_{q^2}\setminus\f_q$, we first choose monic $P_1(X)\in\f_{q^2}[X]$ of degree $k$ such that  $\text{gcd}(P_1,\widetilde{P_1})=1$; the number of choices of such $P_1$, denoted by $\Theta_k$, is determined in Appendix, Lemma~A\ref{LA4}. Next,  let $Q_1(X)=\epsilon X^k\overline{P_1}(X^{-1})$, where $\epsilon\in\f_{q^2}^*$ is such that $Q_1$ is monic. Afterwards, choose $c\in\mu_{q+1}$ and $u\in\f_{q^2}\setminus\f_q$ arbitrarily, and let $t$ and $v$ be uniquely determined by \eqref{2.21}. Hence the total number of $g$ satisfying \eqref{2.13} and $\bar g(X)=g(X^{-1})$ in Case (iii) is
\begin{align}\label{case3}
N(r_1,0)\,
&=(q+1)(q^2-q)\Theta_k\\
&=\frac{q(q^2-1)}{1+q^2}\bigl[(-1)^k(1+q)+q^{2k+1}(q-1)\bigr]\ \text{(by Lemma~A\ref{LA4})}.\nonumber
\end{align}

\medskip
{\bf Step 3.} We have
\[
\text{Fix}(A_{\{\alpha,\alpha^q\}})=\frac 1{|G(\f_q)|}\text{(the number of $g$'s in Step 2)}.
\]

\begin{thm}\label{T2.2}
Let $\alpha\in\f_{q^2}\setminus\f_q$ with $o(\alpha^{q-1})=d$. Then
\[
\text{\rm Fix}(A_{\{\alpha,\alpha^q\}})=\begin{cases}
\displaystyle q^{2n/d-2}+\frac{(q+1)(q^{2n/d}-(-1)^{n/d})}{q^2+1}&\text{if $d\mid n$ and $d$ is even},\vspace{0.2em}\cr
q^{2n/d-2}&\text{if $d\mid n$ and $d$ is odd},\vspace{0.2em}\cr
\displaystyle \frac 1{1+q^2}\bigl[(-1)^{\lfloor n/d\rfloor}(1+q)+q^{2\lfloor n/d\rfloor+1}(q-1)\bigr]&\text{if $d\nmid n$}.
\end{cases}
\]
\end{thm}

\begin{proof}
$1^\circ$ Assume that $d\mid n$ and $d$ is even. By \eqref{case1} and \eqref{case2},
\begin{align*}
\text{Fix}(A_{\{\alpha,\alpha^q\}})\,&=\frac1{q(q^2-1)}\Bigl[(q^2-1)q^{2n/d-1}+\frac{q(q+1)(q^2-1)(q^{2n/d}-(-1)^{n/d})}{q^2+1}\Bigr]\cr
&=q^{2n/d-2}+\frac{(q+1)(q^{2n/d}-(-1)^{n/d})}{q^2+1}.
\end{align*}

\medskip
$2^\circ$ Assume that $d\mid n$ and $d$ is odd. By \eqref{case1} and \eqref{case2},
\[
\text{Fix}(A_{\{\alpha,\alpha^q\}})=\frac1{q(q^2-1)}(q^2-1)q^{2n/d-1}=q^{2n/d-2}.
\]

\medskip
$3^\circ$ Assume that $d\nmid n$. By \eqref{case3},
\begin{align*}
\text{Fix}(A_{\{\alpha,\alpha^q\}})\,&=\frac1{q(q^2-1)}\cdot\frac{q(q^2-1)}{1+q^2}\bigl[(-1)^k(1+q)+q^{2k+1}(q-1)\bigr]\cr
&=\frac 1{1+q^2}\bigl[(-1)^k(1+q)+q^{2k+1}(q-1)\bigr].
\end{align*}
\end{proof}


\section{Determination of $\text{Fix}(B_a)$}

\subsection{A useful lemma}\

Let $p=\text{char}\,\f_q$.
Every $f(X)\in\f_q[X]$ has a representation 
\begin{equation}\label{eq:rep}
f(X)=g_{p-1}(X^p-X)X^{p-1}+g_{p-2}(X^p-X)X^{p-2}+\cdots+g_0(X^p-X),
\end{equation}
where $g_i\in\f_q[X]$. Define $\Delta f=f(X+1)-f(X)$. Then $\Delta^pf=0$, and for $0\le i\le p-1$,
\[
\Delta^if=g_i(X^p-X)i!+\sum_{j=i+1}^{p-1}g_j(X^p-X)\Delta^iX^j.
\]
It follows that $g_i$ in \eqref{eq:rep} are uniquely determined by $f$.

\begin{lem}\label{LA8}
Let $0\le i\le p-1$. Then $\Delta^if=0$ if and only if $g_j=0$ for all $i\le j\le p-1$ in \eqref{eq:rep}.
\end{lem}

\begin{proof}
($\Leftarrow$) Obvious.

\medskip
($\Rightarrow$) Assume the contrary. Let $j_0$ be the largest $j$ such that $g_j\ne 0$. Then $i\le j_0\le p-1$. We have
\begin{align*}
\Delta^if\,&=g_{j_0}(X^p-X)\Delta^iX^{j_0}+\sum_{j<j_0}g_j(X^p-X)\Delta^iX^j\cr
&=g_{j_0}(X^p-X)\binom{j_0}iX^{j_0-i}+\sum_{j<j_0-i}h_j(X^p-X)X^j\kern2.5em(h_j\in\f_q[X])\cr
&\ne 0,
\end{align*}
which is a contradiction.
\end{proof}

\subsection{Determination of $\text{Fix}(B_a)$}\

Recall that $B_a=\left[\begin{smallmatrix}a&a\cr 0&a\end{smallmatrix}\right]$, $a\in\f_q^*$, so $\phi_{B_a}=X+1$. Let $F=\f_q(P/Q)$, where $P,Q\in\f_q[X]$ are monic, $\deg P=n>\deg Q$, and $\text{gcd}(P,Q)=1$. Then $B_a(F)=\f_q(P(X+1)/Q(X+1))$. Hence $B_a(F)=F$ if and only if
\begin{equation}\label{2.5-1}
\begin{cases}
Q(X+1)=Q(X),\cr
P(X+1)=P(X)+cQ(X)\quad \text{for some}\ c\in\f_q.
\end{cases}
\end{equation}

\medskip
{\bf Case 1.} Assume $c=0$. Then \eqref{2.5-1} holds if and only if $P(X)=P_1(X^p-X)$, $Q(X)=Q_1(X^p-X)$, where $P_1,Q_1\in\f_q[X]$ are such that $\deg P_1=n/p>\deg Q_1$ (must have $p\mid n$) and $\text{gcd}(P_1,Q_1)=1$. The number of such $(P,Q)$ is $\alpha_{n/p}$.

\medskip
{\bf Case 2.} Assume $c\ne 0$. Then \eqref{2.5-1} holds if and only if
\begin{equation}\label{2.5-2}
\begin{cases}
Q=c^{-1}\Delta P,\cr
\Delta^2 P=0,\cr
\text{gcd}(P(X),P(X+1))=1.
\end{cases}
\end{equation}
Condition \eqref{2.5-2} is equivalent to
\begin{equation}\label{2.5-3}
\begin{cases}
\Delta^2P=0,\ \Delta P\ne 0,\ \text{gcd}(P(X),P(X+1))=1,\cr
Q=c^{-1}\Delta P,\ \text{where $c$ is uniquely determined by $P$}.
\end{cases}
\end{equation}
By Lemma~\ref{LA8}, the $P(X)$ in \eqref{2.5-3} has the form
\[
P(X)=P_1(X^p-X)X+P_0(X^p-X),
\]
where $P_1\ne 0$. Since $P(X+1)-P(X)=P_1(X^p-X)$, $\text{gcd}(P(X),P(X+1))=1$ if and only if $\text{gcd}(P_0,P_1)=1$. Also note that 
\[
\deg P=\max\{p\deg P_1+1,\,p\deg P_0\}.
\]
Hence the number of $(P,Q)$ satisfying \eqref{2.5-3} is
\[
\begin{cases}
(q-1)\alpha_{n/p}&\text{if}\ n\equiv 0\pmod p,\cr
q&\text{if}\ n=1,\cr
(q-1)(\alpha_{(n-1)/p}+\alpha_{(n-1)/p,(n-1)/p})&\text{if}\ n\equiv 1\pmod p,\ n>1,\cr
0&\text{otherwise}.
\end{cases}
\]
Therefore,
\[
\text{Fix}(B_a)=\begin{cases}
\displaystyle\frac 1q(\alpha_{n/p}+(q-1)\alpha_{n/p})&\text{if}\ n\equiv 0\pmod p,\vspace{0.3em}\cr
1&\text{if}\ n=1,\vspace{0.3em}\cr
\displaystyle\frac {q-1}q(\alpha_{(n-1)/p}+\alpha_{(n-1)/p,(n-1)/p})&\text{if}\ n\equiv 1\pmod p,\ n>1,\vspace{0.3em}\cr
0&\text{otherwise}.
\end{cases}
\]
Recall that $\alpha_i$ and $\alpha_{i,j}$ are given by Lemma~A\ref{LA1}. When $n\equiv 0\pmod p$,
\[
\text{Fix}(B_a)=\alpha_{n/p}=q^{2n/p-1}.
\]
When $n\equiv 1\pmod p$ and $n>1$,
\[
\text{Fix}(B_a)=\frac{q-1}q(q^{2(n-1)/p-1}+q^{2(n-1)/p}(1-q^{-1}))=q^{2(n-1)/p-1}(q-1).
\]
To summarise,
\begin{equation}\label{FixBa}
\text{Fix}(B_a)=\begin{cases}
q^{2n/p-1}&\text{if}\ n\equiv 0\pmod p,\cr
1&\text{if}\ n=1,\cr
q^{2(n-1)/p-1}(q-1)&\text{if}\ n\equiv 1\pmod p,\ n>1,\cr
0&\text{otherwise}.
\end{cases}
\end{equation}


\section{The Main Theorem}

\begin{thm}\label{T-main} 
For $n\ge 1$, we have 
\begin{equation}\label{6.1}
\frak N(q,n)=\frac{q^{2n-3}}{q^2-1}+\frac 1{2(q-1)}\frak A(q,n)+\frac 1{2(q+1)}\frak B(q,n)+\frac 1q\frak C(q,n),
\end{equation}
where
\begin{equation}\label{A(q,n)}
\frak A(q,n)=\sum_{\substack{1<d\,\mid\, q-1\cr d\,\mid\, n}}\phi(d)\Bigl(q^{2n/d-2}+\frac{(d-1)(q^{2n/d}-1)}{q+1}\Bigr)+\sum_{\substack{1<d\,\mid\, q-1\cr d\,\nmid\, n}}\phi(d)\frac{q^{2\lfloor n/d\rfloor+1}+1}{q+1},
\end{equation}
\begin{align}\label{B(q,n)}
\frak B(q,n)=\,&\sum_{\substack{d\; \text{\rm even}\cr d\,\mid\,\text{\rm gcd}(q+1,n)}}\phi(d)\Bigl(q^{2n/d-2}+\frac{(q+1)(q^{2n/d}-(-1)^{n/d})}{q^2+1}\Bigr)\\
&+\sum_{\substack{d\; \text{\rm odd}\cr 1<d\,\mid\,\text{\rm gcd}(q+1,n)}}\phi(d)q^{2n/d-2}\cr
&+\frac 1{q^2+1}\sum_{\substack{d\,\mid\, q+1\cr d\,\nmid\, n}}\phi(d)\bigl((-1)^{\lfloor n/d\rfloor}(1+q)+q^{2\lfloor n/d\rfloor+1}(q-1)\bigr),\nonumber
\end{align}
\begin{equation}\label{C(q,n)}
\frak C(q,n)=\begin{cases}
q^{2n/p-1}&\text{if}\ n\equiv 0\pmod p,\cr
1&\text{if}\ n=1,\cr
q^{2(n-1)/p-1}(q-1)&\text{if}\ n\equiv 1\pmod p,\ n>1,\cr
0&\text{otherwise}.
\end{cases}
\end{equation}
In \eqref{A(q,n)} and \eqref{B(q,n)}, $\phi$ is the Euler function.
\end{thm}

\begin{proof}
We have
\begin{align*}
\frak N(q,n)=\,&\frac 1{q(q-1)^2(q+1)}\sum_{a\in\f_q^*}\text{Fix}(A_a)+\frac 1{(q-1)^2}\sum_{\substack{\{a,b\}\subset\f_q^*\cr a\ne b}}\text{Fix}(A_{\{a,b\}})\cr
&+\frac 1{q^2-1}\sum_{\{\alpha,\alpha^q\}\subset\f_{q^2}\setminus\f_q}\text{Fix}(A_{\{\alpha,\alpha^q\}})+\frac 1{q(q-1)}\sum_{a\in\f_q^*}\text{Fix}(B_a).
\end{align*}
We now compute the four sums in the above. 

\medskip
$1^\circ$ We have
\[
\sum_{a\in\f_q^*}\text{Fix}(A_a)=(q-1)q^{2(n-1)}.
\]

\medskip
$2^\circ$ We have
\begin{align*}
&\sum_{\substack{\{a,b\}\subset\f_q^*\cr a\ne b}}\text{Fix}(A_{\{a,b\}})=\frac 12\sum_{a\in\f_q^*}\sum_{b\in\f_q^*\setminus\{1\}}\text{Fix}(A_{\{ab,a\}})=\frac{q-1}2\sum_{b\in\f_q^*\setminus\{1\}}\text{Fix}(A_{\{b,1\}})\cr
&=\frac{q-1}2\Bigl[\sum_{\substack{1<d\,\mid\, q-1\cr d\,\mid\, n}}\phi(d)\Bigl(q^{2n/d-2}+\frac{(d-1)(q^{2n/d}-1)}{q+1}\Bigr)+\sum_{\substack{1<d\,\mid\, q-1\cr d\,\nmid\, n}}\phi(d)\frac{(q^{2\lfloor n/d\rfloor}+1)}{q+1}\Bigr]\cr
&\kern28em \text{(by Theorem~\ref{T2.1})}\cr
&=\frac{q-1}2\frak A(q,n).
\end{align*}

\medskip
$3^\circ$ By Theorem~\ref{T2.2}, 
\begin{align*}
&\sum_{\{\alpha,\alpha^q\}\subset\f_{q^2}\setminus\f_q}\text{Fix}(A_{\{\alpha,\alpha^q\}})\cr
&=\frac{q-1}2\Bigl[
\sum_{\substack{d\; \text{even}\cr d\,\mid\,\text{gcd}(q+1,n)}}\phi(d)\Bigl(q^{2n/d-2}+\frac{(q+1)(q^{2n/d}-(-1)^{n/d})}{q^2+1}\Bigr)\cr
&+\sum_{\substack{d\; \text{odd}\cr 1<d\,\mid\,\text{gcd}(q+1,n)}}\phi(d)q^{2n/d-2}\cr
&+\frac 1{q^2+1}\sum_{\substack{d\,\mid\, q+1\cr d\,\nmid\, n}}\phi(d)\bigl((-1)^{\lfloor n/d\rfloor}(1+q)+q^{2\lfloor n/d\rfloor+1}(q-1)\bigr)
\Big]\cr
&=\frac{q-1}2\frak B(q,n).
\end{align*}

\medskip
$4^\circ$ By \eqref{FixBa},
\[
\sum_{a\in\f_q^*}\text{Fix}(B_a)=(q-1)\frak C(q,n).
\]
\end{proof}


\section{$\frak N(q,n)$ for Small n}

\subsection{$n=1$}\

We have
\[
\frak A(q,1)=\sum_{1<d\,\mid\, q-1}\phi(d)=q-2,
\]
\[
\frak B(q,1)=\frac 1{q^2+1}\sum_{1<d\,\mid\, q+1}\phi(d)\bigl[(1+q)+q(q-1)\bigr]=\sum_{1<d\,\mid\, q+1}\phi(d)=q+1-1=q,
\]
\[
\frak C(q,1)=1.
\]
Hence
\[
\frak N(q,1)=\frac{q^{-1}}{q^2-1}+\frac 1{2(q-1)}(q-2)+\frac 1{2(q+1)}q+\frac 1q=1,
\]
as expected.

\subsection{$n=2$}\ 

{\bf Case 1.} Assume $q$ is even. We have
\[
\frak A(q,2)=\sum_{1<d\,\mid\, q-1}\phi(d)=q-2,
\]
\[
\frak B(q,2)=\frac 1{q^2+1}\sum_{\substack{1<d\,\mid\, q+1\cr d\,\nmid\, 2}}\phi(d)\bigl[(1+q)+q(q-1)\bigr]=\sum_{1<d\,\mid\, q+1}\phi(d)=q+1-1=q,
\]
\[
\frak C(q,2)=q.
\]
Hence
\[
\frak N(q,2)=\frac q{q^2-1}+\frac 1{2(q-1)}(q-2)+\frac 1{2(q+1)}q+\frac 1q q=2.
\]

Since $X^2$ and $X^2+X$ are nonequivalent ($X^2$ is a permutation of $\Bbb P^1(\f_q)$ but $X^2+X$ is not), 
\[
X^2,\ X^2+X
\]
is a list of representatives of the equivalence classes of rational functions of degree 2 over $\f_q$.

\medskip
{\bf Case 2.} Assume $q$ is odd. We have
\[
\frak A(q,2)=\phi(2)\Bigl(1+\frac{q^2-1}{q+1}\Bigr)+\sum_{2<d\,\mid\, q-1}\phi(d)=q+q-1-2=2q-3,
\]
\begin{align*}
\frak B(q,2)=\,&\phi(2)\Bigl(1+\frac{(q+1)(q^2+1)}{q^2+1}\Bigr)+\frac 1{q^2+1}\sum_{\substack{d\,\mid\, q+1\cr d\,\nmid\, 2}}\phi(d)\bigl((1+q)+q(q-1)\bigr)\cr
=\,&q+2+\sum_{2<d\,\mid\, q+1}\phi(d)=q+2+q+1-2=2q+1,
\end{align*}
\[
\frak C(q,2)=0.
\]
Hence
\[
\frak N(q,2)=\frac q{q^2-1}+\frac 1{2(q-1)}(2q-3)+\frac 1{2(q+1)}(2q+1)=2.
\]

In this case, a list of representatives of the equivalence classes of rational functions of degree 2 over $\f_q$ is given by 
\[
X^2,\ \frac{X^2+b}X,
\]
where $b$ is any fixed nonsquare of $\f_q$. 

\begin{proof}
It suffices to show that every $f\in\f_q(X)$ of degree 2 is equivalent to one of the above two rational functions.

If $f$ is a polynomial, then $f\sim X^2$.

If $f$ is not a polynomial, then $f\sim(X^2+aX+b)/X$, where $b\in\f_q^*$. Thus $f\sim(X^2+b)/X$. If $b=c^2$ for some $c\in\f_q^*$, then 
\begin{align*}
f\,&\sim\frac{X^2+2cX+c^2}X=\frac{(X+c)^2}X\sim\frac X{(X+c)^2}\sim\frac{X-c}{X^2}= \frac 1X-c\Bigl(\frac 1X\Bigr)^2\cr
&\sim X-cX^2\sim X^2.
\end{align*}
\end{proof} 

\subsection{$n=3$}\

$1^\circ$ Computing $\frak A(q,3)$.

\medskip
First assume $q$ is even.

If $q-1\equiv 0\pmod 3$, 
\begin{align*}
\frak A(q,3)\,&=\phi(3)\Bigl(1+\frac{2(q^2-1)}{q+1}\Bigr)+\sum_{\substack{1<d\,\mid\, q-1\cr d\,\nmid\, 3}}\phi(d)\cr
&=2(1+2(q-1))+q-1-\phi(1)-\phi(3)\cr
&=2(2q-1)+q-1-3=5q-6.
\end{align*}

If $q-1\not\equiv 0\pmod 3$,
\[
\frak A(q,3)=\sum_{\substack{1<d\,\mid\, q-1\cr d\,\nmid\, 3}}\phi(d)=q-1-\phi(1)=q-2.
\]

\medskip
Next, assume $q$ is odd.

If $q-1\equiv 0\pmod 3$,
\begin{align*}
\frak A(q,3)\,&=\phi(3)\Bigl(1+\frac {2(q^2-1)}{q+1}\Bigr)+\phi(2)\frac{q^3+1}{q+1}+\sum_{3<d\,\mid\, q-1}\phi(d)\cr
&=2(1+2(q-1))+q^2-q+1+q-1-\phi(1)-\phi(2)-\phi(3)\cr
&=2(2q-1)+q^2-4=q^2+4q-6.
\end{align*}

If $q-1\not\equiv 0\pmod 3$,
\[
\frak A(q,3)=\phi(2)\frac{q^3+1}{q+1}+\sum_{3<d\,\mid\, q-1}\phi(d)=q^2-q+1+q-1-\phi(1)-\phi(2)=q^2-2.
\]

To summarize,
\[
\frak A(q,3)=\begin{cases}
5q-6&\text{if}\ q\equiv 4\pmod 6,\cr
q-2&\text{if}\ q\equiv 2\pmod 6,\cr
q^2+4q-6&\text{if}\ q\equiv 1\pmod 6,\cr
q^2-2&\text{if}\ q\equiv 3,5\pmod 6.
\end{cases}
\]

\medskip
$2^\circ$ Computing $\frak B(q,3)$.

First assume $q$ is even.

If $q+1\equiv 0\pmod 3$,
\begin{align*}
\frak B(q,3)\,&=\phi(3)+\frac 1{q^2+1}\sum_{\substack{d\,\mid\, q+1\cr d\,\nmid\, 3}}\phi(d)\bigl[(1+q)+q(q-1)\bigr]\cr
&=2+\sum_{\substack{d\,\mid\, q+1\cr d\,\nmid\, 3}}\phi(d)=2+q+1-\phi(1)-\phi(3)=q.
\end{align*}

If $q+1\not\equiv 0\pmod 3$,
\[
\frak B(q,3)=\frac 1{q^2+1}\sum_{\substack{d\,\mid\, q+1\cr d\,\nmid\, 3}}\phi(d)\bigl[(1+q)+q(q-1)\bigr]=\sum_{\substack{d\,\mid\, q+1\cr d\,\nmid\, 3}}\phi(d)=q+1-\phi(1)=q.
\]

\medskip
Next, assume $q$ is odd.

If $q+1\equiv 0\pmod 3$,
\begin{align*}
\frak B(q,3)\,&=\phi(3)+\frac 1{q^2+1}\Bigl[\phi(2)(-(1+q)+q^3(q-1))+\sum_{3<d\,\mid\, q+1}\phi(d)(1+q+q(q-1))\Bigr]\cr
&=2+\frac 1{q^2+1}\Bigl[q^4-q^3-q-1+(q^2+1)\sum_{3<d\,\mid\, q+1}\phi(d)\Bigr]\cr
&=2+\frac 1{q^2+1}\bigl[(q^2+1)(q^2-q-1)+(q^2+1)(q+1-\phi(1)-\phi(2)-\phi(3))\bigr]\cr
&=2+q^2-q-1+q+1-4=q^2-2.
\end{align*}

If $q+1\not\equiv 0\pmod 3$,
\begin{align*}
\frak B(q,3)\,&=\frac 1{q^2+1}\Bigl[\phi(2)(-(1+q)+q^3(q-1))+\sum_{3<d\,\mid\, q+1}\phi(d)((1+q)+q(q-1))\Bigr]\cr
&=\frac 1{q^2+1}\bigl[(q^2+1)(q^2-q-1)+(q^2+1)(q+1-\phi(1)-\phi(2))\bigr]\cr
&=q^2-q-1+q+1-2=q^2-2.
\end{align*}

To summarize,
\[
\frak B(q,3)=\begin{cases}
q&\text{if $q$ is even},\cr
q^2-2&\text{if $q$ is odd}.
\end{cases}
\]

\medskip
$3^\circ$ Computing $\frak C(q,3)$. We have
\[
\frak C(q,3)=\begin{cases}
q(q-1)&\text{if}\ p=2,\cr
q&\text{if}\ p=3,\cr
0&\text{otherwise}.
\end{cases}
\]

\medskip
$4^\circ$ Computing $\frak N(q,3)$.

\medskip
If $q\equiv 1\pmod 6$,
\[
\frak N(q,3)=\frac{q^3}{q^2-1}+\frac 1{2(q-1)}(q^2+4q-6)+\frac 1{2(q+1)}(q^2-2)=2(q+1).
\]

\medskip
If $q\equiv 2\pmod 6$,
\[
\frak N(q,3)=\frac{q^3}{q^2-1}+\frac 1{2(q-1)}(q-2)+\frac 1{2(q+1)}q+\frac 1qq(q-1)=2q.
\]

\medskip
If $q\equiv 3\pmod 6$, i.e., $p=3$,
\[
\frak N(q,3)=\frac{q^3}{q^2-1}+\frac 1{2(q-1)}(q^2-2)+\frac 1{2(q+1)}(q^2-2)+\frac 1qq=2q+1.
\]

\medskip
If $q\equiv 4\pmod 6$,
\[
\frak N(q,3)=\frac{q^3}{q^2-1}+\frac 1{2(q-1)}(5q-6)+\frac 1{2(q+1)}q+\frac 1qq(q-1)=2(q+1).
\]

\medskip
If $q\equiv 5\pmod 6$,
\[
\frak N(q,3)=\frac{q^3}{q^2-1}+\frac 1{2(q-1)}(q^2-2)+\frac 1{2(q+1)}(q^2-2)=2q.
\]

\medskip
To summarize,
\[
\frak N(q,3)=\begin{cases}
2(q+1)&\text{if}\ q\equiv 1,4\pmod 6,\cr
2q&\text{if}\ q\equiv 2,5\pmod 6,\cr
2q+1&\text{if}\ q\equiv 3\pmod 6.
\end{cases}
\]

As mentioned in Section~1, rational functions of degree $3$ in $\f_q(X)$ have been classified for even $n$ \cite{Mattarei-Pizzato-arXiv2104.00111}; for odd $q$, the question is still open.


\subsection{$n=4$}\

We include the formulas for $\frak A(q,4)$, $\frak B(q,4)$, $\frak C(q,4)$ and $\frak N(q,4)$ but omit the details of the computations.
\[
\frak A(q,4)=\begin{cases}
-2-q+2q^2&\text{if}\ q\equiv 4,10\pmod{12},\cr
-2+q&\text{if}\ q\equiv 2,8\pmod{12},\cr
-10+6q+2q^2+q^3&\text{if}\ q\equiv 1\pmod{12},\cr
-10+8q+q^3&\text{if}\ q\equiv 5,9\pmod{12},\cr
-4+2q^2+q^3&\text{if}\ q\equiv 7\pmod{12},\cr
-4+2q+q^3&\text{if}\ q\equiv 3,11\pmod{12}.
\end{cases}
\]

\[
\frak B(q,4)=\begin{cases}
-2+q^2&\text{if}\ q\equiv 2,8\pmod{12},\cr
q&\text{if}\ q\equiv 4\pmod{12},\cr
-4+4q^2+q^3&\text{if}\ q\equiv 11\pmod{12},\cr
2q+2q^2+q^3&\text{if}\ q\equiv 3,7\pmod{12},\cr
-6-2q+4q^2+q^3&\text{if}\ q\equiv 5\pmod{12},\cr
-2+2q^2+q^3&\text{if}\ q\equiv 1,9\pmod{12}.
\end{cases}
\]

\[
\frak C(q,4)=\begin{cases}
q^3&\text{if}\ p=2,\cr
q(q-1)&\text{if}\ p=3,\cr
0&\text{otherwise}.
\end{cases}
\]

\[
\frak N(q,4)=\begin{cases}
4+3q+q^2+q^3&\text{if}\ q\equiv 1\pmod{12},\vspace{0.2em}\cr
\displaystyle\frac 32q+q^2+q^3&\text{if}\ q\equiv 2,8\pmod{12},\vspace{0.2em}\cr
1+3q+q^2+q^3&\text{if}\ q\equiv 3\pmod{12},\cr
1+2q+q^2+q^3&\text{if}\ q\equiv 4\pmod{12},\cr
2+3q+q^2+q^3&\text{if}\ q\equiv 5,7\pmod{12},\cr
3+3q+q^2+q^3&\text{if}\ q\equiv 9\pmod{12},\cr
3q+q^2+q^3&\text{if}\ q\equiv 11\pmod{12}.
\end{cases}
\]


\section{Equivalence Classes of Polynomials}

\begin{lem}\label{L8.1}
Let $f,g\in\f_q[X]\setminus\f_q$. Then $g=\phi\circ f\circ \psi$ for some $\phi,\psi\in G(\f_q)$ if and only if $g=\alpha\circ f\circ\beta$ for some $\alpha,\beta\in\text{\rm AGL}(1,\f_q)$.
\end{lem}

\begin{proof}
($\Rightarrow$) Let $\psi(X)=A(X)/B(X)$. 

\medskip
{\bf Case 1.} Assume that $B(X)=1$. Then $\psi=A\in\text{AGL}(1,\f_q)$. Since $f\circ A=f(A(X))\in\f_q[X]$ and $\phi\circ f\circ A\in\f_q[X]$, it follows that $\phi\in \text{AGL}(1,\f_q)$.

\medskip
{\bf Case 2.} Assume that $B(X)\notin\f_q$. Let $B(X)=X+d$ and $A(X)=aX+b$. Let $f(X)=X^n+a_{n-1}X^{n-1}+\cdots+a_0$. Then 
\[
f(\phi(X))=\frac{A(X)^n+a_{n-1}A(X)^{n-1}B(X)+\cdots+a_0B(X)^n}{B(X)^n}.
\]
Let $\phi(X)=(sX+t)/(uX+v)$. Then
\begin{equation}\label{u...=1}
u\bigl(A(X)^n+a_{n-1}A(X)^{n-1}B(X)+\cdots+a_0B(X)^n \bigr)+vB(X)^n=1
\end{equation}
and 
\[
g(X)=s\bigl(A(X)^n+a_{n-1}A(X)^{n-1}B(X)+\cdots+a_0B(X)^n \bigr)+tB(X)^n.
\]
By \eqref{u...=1}, $u\ne 0$ and
\[
g(X)=su^{-1}(1-vB(X)^n)+tB(X)^n=su^{-1}+(t-su^{-1}v)B(X)^n.
\]
Hence we may assume $g(X)=X^n$. By \eqref{u...=1} again,
\begin{align*}
uf\Bigl(\frac{A(X)}{B(X)}\Bigr)+v\,&=\frac 1{B(X)^n}=\Bigl(\frac 1{X+d}\Bigr)^n=\Bigl(\frac{AX+b}{X+d}-a\Bigr)^n(b-ad)^{-n}\cr
&=\Bigl(\frac{A(X)}{B(X)}-a\Bigr)^n(b-ad)^{-n}.
\end{align*}
So $f(X)=u^{-1}(b-ad)^{-n}(X-a)^n-u^{-1}v$. Hence we may assume $f(X)=X^n$. Then $f=g$.
\end{proof}

Because of Lemma~\ref{L8.1}, we define
two polynomials $f,g\in\f_q[X]\setminus\f_q$ to be {\em equivalent} if there exist $\alpha,\beta\in\text{AGL}(1,\f_q)$ such that $g=\alpha\circ f\circ \beta$; the meaning of equivalence between $f$ and $g$ is the same whether they are treated as polynomials or as rational functions. 

Let
\[
\mathcal P_{q,n}=\{f\in\f_q[X]:\deg f=n\}
\]
and let $\frak M(q,n)$ denote the number of equivalence classes in $\mathcal P_{q,n}$. Compared with $\frak N(q,n)$, $\frak M(q,n)$ is much easier to determine.

For $f,g\in\f_q[X]\setminus\f_q$, define $f\overset L\sim g$ if there exists $\alpha\in\text{AGL}(1,\f_q)$ such that $g=\alpha\circ f$. Let $[f]$ denote the $\overset L\sim$ equivalence class of $f$. Each $\overset L\sim$ equivalence class has a unique representative $X^n+a_{n-1}X^{n-1}+\cdots+a_1X$. Let $\text{AGL}(1,\f_q)$ act on the set of $\overset L\sim$ equivalence classes in $\f_q[X]\setminus\f_q$ as follows: For $f\in\f_q[X]\setminus\f_q$ and $\alpha\in\text{AGL}(1,\f_q)$, $[f]^\alpha=[f\circ\alpha]$. Then $\frak M(q,n)$ is the number of $\text{AGL}(1,\f_q)$-orbits in $\Omega_n:=\{[f]:f\in\mathcal P_{q,n}\}$. The information about the conjugacy classes of $\text{AGL}(1,\f_q)$ is given in Table~\ref{Tb3}. For $\alpha\in\text{AGL}(1,\f_q)$, let 
$\text{Fix}(\alpha)$ be the number of elements in $\Omega_n$ fixed by $\alpha$. All we have to do is to determine $\text{Fix}(\alpha)$ for each representative $\alpha$ in Table~\ref{Tb3}.

\begin{table}[h]
\caption{Conjugacy classes of $\text{AGL}(1,\f_q)$}\label{Tb3}
   \renewcommand*{\arraystretch}{1.2}
    \centering
     \begin{tabular}{c|c}
     \hline
     	 representative & size of the centralizer \\ \hline      
         $X$ & $q(q-1)$  \\     
         $aX,\ a\in\f_q^*,\ a\ne 1$ & $q-1$  \\
         $X+1$ & $q$  \\
         \hline
    \end{tabular}
\end{table}

Clearly,
\begin{equation}\label{FixX}
\text{Fix}(X)=q^{n-1}.
\end{equation}

\medskip
Next, we compute $\text{Fix}(aX)$, where $a\in\f_q^*$, $a\ne 1$. Let $o(a)=d$. Then $[f]\in\Omega_n$ is fixed by $aX$ if and only if 
\[
f\overset L\sim X^rh(X^d),
\]
where $0\le r<d$, $n\equiv r\pmod d$, $h\in\f_q[X]$ is monic of degree $(n-r)/d$, and $h(0)=0$ if $r=0$. Thus
\begin{align}\label{FixaX}
\text{Fix}(aX)\,&=\begin{cases}
q^{n/d-1}&\text{if}\ d\mid n\cr
q^{\lfloor n/d\rfloor}&\text{if}\ d\nmid n
\end{cases}\\
&=q^{\lceil n/d\rceil-1}.\nonumber
\end{align}

\medskip
Now we comput $\text{Fix}(X+1)$. For $[f]\in\Omega_n$,
\[
\begin{array}{cl}
& \text{$[f]$ is fixed by $X+1$}\vspace{0.2em}\cr
\Leftrightarrow & f(X+1)=f(X)+a,\ \text{where}\ a\in\f_q \vspace{0.2em}\cr
\Leftrightarrow & f(X)=g(X)+aX,\ \text{where}\  a\in\f_q,\ g\in\f_q[X],\ \Delta g=0 \vspace{0.2em}\cr
\Leftrightarrow & f(X)=h(X^p-X)+aX,\ \text{where}\ a\in\f_q,\ h\in\f_q[X],\ p=\text{char}\,\f_q.
\end{array}
\]
In the above, we may assume that $f$ is monic and $f(0)=0$. Therefore,
when $p\mid n$, $h$ is of degree $n/p$ with $h(0)=0$; when $p\nmid n$, $h=0$, $n=1$ and $a=1$. So,
\begin{equation}\label{FixX+1}
\text{Fix}(X+1)=\begin{cases}
q^{n/p-1}\cdot q=q^{n/p}&\text{if}\ p\mid n,\cr
1&\text{if}\ n=1,\cr
0&\text{if $p\nmid n$ and $n>1$}.
\end{cases}
\end{equation}

\begin{thm}\label{T3.1}
Let $p=\text{\rm char}\,\f_q$. We have
\[
\frak M(q,n)=\frac{q^{n-2}}{q-1}+\frac 1{q-1}\sum_{1<d\,\mid\, q-1}\phi(d)q^{\lceil n/d\rceil-1}+\begin{cases}
q^{n/p-1}&\text{if}\ p\mid n,\cr
q^{-1}&\text{if}\ n=1,\cr
0&\text{if $p\nmid n$ and $n>1$}.
\end{cases}
\]
\end{thm}

\begin{proof}
By Burnside's lemma and \eqref{FixX} -- \eqref{FixX+1},
\begin{align*}
\frak M(q,n)\,&=\frac 1{q(q-1)}\text{Fix}(X)+\frac 1{q-1}\sum_{a\in\f_q^*\setminus\{1\}}\text{Fix}(aX)+\frac 1q\text{Fix}(X+1)\cr
&=\frac{q^{n-2}}{q-1}+\frac 1{q-1}\sum_{1<d\,\mid\, q-1}\phi(d)q^{\lceil n/d\rceil-1}+\frac 1q\text{Fix}(X+1),
\end{align*}
where
\[
\frac 1q\text{Fix}(X+1)=\begin{cases}
q^{n/p-1}&\text{if}\ p\mid n,\cr
q^{-1}&\text{if}\ n=1,\cr
0&\text{if $p\nmid n$ and $n>1$}.
\end{cases}
\]
\end{proof}

In Theorem~\ref{T3.1}, we can write 
\begin{align*}
&\frac{q^{n-2}}{q-1}+\frac 1{q-1}\sum_{1<d\,\mid\, q-1}\phi(d)q^{\lceil n/d\rceil-1}\cr
=\,&\frac{q^{n-2}}{q-1}+\frac 1{q-1}\Bigl(\sum_{d\,\mid\, q-1}\phi(d)q^{\lceil n/d\rceil-1}-q^{n-1}\Bigr)\cr
=\,&\frac 1{q-1}\sum_{d\,\mid\, q-1}\phi(d)q^{\lceil n/d\rceil-1}+\frac{q^{n-2}-q^{n-1}}{q-1}\cr
=\,&\frac 1{q-1}\Bigl(\sum_{d\,\mid\, q-1}\phi(d)(q^{\lceil n/d\rceil-1}-1)+\sum_{d\,\mid\, q-1}\phi(d)\Bigr)-q^{n-2}\cr
=\,&\frac 1{q-1}\sum_{\substack{d\,\mid\, q-1\cr d<n}}\phi(d)(q^{\lceil n/d\rceil-1}-1)+1-q^{n-2}.
\end{align*}
Hence
\[
\frak M(q,n)=\frac 1{q-1}\sum_{\substack{d\,\mid\, q-1\cr d<n}}\phi(d)(q^{\lceil n/d\rceil-1}-1)+\begin{cases}
1-q^{n-2}+q^{n/p-1}&\text{if}\ p\mid n,\cr
1&\text{if}\ n=1,\cr
1-q^{n-2}&\text{if $p\nmid n$ and $n>1$}.
\end{cases}
\]
In the above, the sum
\[
\frac 1{q-1}\sum_{\substack{d\,\mid\, q-1\cr d<n}}\phi(d)(q^{\lceil n/d\rceil-1}-1)
\]
can be made more explicit as follows: Write
\[
\text{lcm}\{1,2,\dots,n-1\}=\prod_{\text{$r$ prime}}r^{\nu_r},\qquad \nu_r=\lfloor\log_r(n-1)\rfloor,
\]
and
\[
\text{gcd}(\text{lcm}\{1,2,\dots,n-1\},q-1)=\prod_{\text{$r$ prime}}r^{u_r}. 
\]
Then
\begin{align*}
&\frac 1{q-1}\sum_{\substack{d\,\mid\, q-1\cr d<n}}\phi(d)(q^{\lceil n/d\rceil-1}-1)\cr
=\,&\sum_{\substack{e_r\le u_r\cr \prod_r r^{e_r}\le n-1}}\phi\Bigl(\prod_r r^{e_r}\Bigr)(q^{\lceil n/\prod_r r^{e_r}\rceil-1}-1)\cr
=\,&\sum_{\substack{e_r\le u_r\cr \prod_r r^{e_r}\le n-1}}\Bigl(\prod_r r^{e_r}\Bigr)\Bigl(\prod_{r: e_r>0}(1-r^{-1})\Bigr)(q^{\lceil n/\prod_r r^{e_r}\rceil-1}-1).
\end{align*}

As concrete examples, we include the formulas for $\frak M(q,n)$ with $1\le n\le 5$.
\[
\frak M(q,1)=1.
\]
\[
\frak M(q,2)=\begin{cases}
2&\text{if}\ p=2,\cr
1&\text{if}\ p>2.
\end{cases}
\]
\[
\frak M(q,3)=\begin{cases}
2&\text{if}\ p=2,\cr
4&\text{if}\ p=3,\cr
3&\text{if}\ p>3.
\end{cases}
\]
\[
\frak M(q,4)=\begin{cases}
q+5&\text{if}\ q\equiv 1\pmod 6,\cr
2q+2&\text{if}\ q\equiv 2\pmod 6,\cr
q+3&\text{if}\ q\equiv 3,5\pmod 6,\cr
2q+4&\text{if}\ q\equiv 4\pmod 6.
\end{cases}
\]
\[
\frak M(q,5)=\begin{cases}
q^2+2q+8&\text{if}\ q\equiv 1\pmod{12}\ \text{and}\ p=5,\cr
q^2+2q+7&\text{if}\ q\equiv 1\pmod{12}\ \text{and}\ p\ne 5,\cr
q^2+q+2&\text{if}\ q\equiv 2,8\pmod{12},\cr
q^2+2q+3&\text{if}\ q\equiv 3,11\pmod{12},\cr
q^2+q+4&\text{if}\ q\equiv 4\pmod{12},\cr
q^2+2q+6&\text{if}\ q\equiv 5\pmod{12}\ \text{and}\ p=5,\cr
q^2+2q+5&\text{if}\ q\equiv 5,7,9\pmod{12}\ \text{and}\ p\ne 5.
\end{cases}
\]

With $\frak M(q,n)$ known, it is not difficult to classify polynomials of low degree over $\f_q$. Tables~\ref{Tb-q1} -- \ref{Tb-q5} give the representatives of the equivalence classes in $\mathcal P_{q,n}$ for $1\le n\le 5$. In each of these cases, it is easy to verify that every $f\in\mathcal P_{q,n}$ is equivalent to one of the representatives, and since their total number equals $\frak M(q,n)$, the representatives are pairwise nonequivalent. In these tables, $\mathcal C_i$ denotes a system of representatives of the cosets of $\{x^i:x\in\f_q^*\}$ in $\f_q^*$. 

\begin{table}[h]
\caption{Equivalence classes of $\mathcal P_{q,1}$}\label{Tb-q1}
   \renewcommand*{\arraystretch}{1.2}
    \centering
     \begin{tabular}{c|c}
     \hline
     	 representative & number \\ \hline      
         $X$ & $1$  \\ \cline{2-2}   
         & $1$  \\
         \hline
    \end{tabular}
\end{table}

\begin{table}[h]
\caption{Equivalence classes of $\mathcal P_{q,2}$}\label{Tb-q2}
   \renewcommand*{\arraystretch}{1.2}
    \centering
     \begin{tabular}{c|c|c}
     \hline
     	$q$ & representative & number \\ \hline      
        even & $X^2+X$ & $1$  \\ 
        & $X^2$ & $1$ \\ \cline{3-3}   
        & & $2$  \\
         \hline
        odd & $X^2$ & $1$  \\ \cline{3-3}   
        & & $1$  \\
         \hline 
    \end{tabular}
\end{table}

\begin{table}[h]
\caption{Equivalence classes of $\mathcal P_{q,3}$}\label{Tb-q3}
   \renewcommand*{\arraystretch}{1.2}
    \centering
     \begin{tabular}{c|c|c}
     \hline
     	$q$ & representative & number \\ \hline      
        $p=2$ & $X^3+X$ & $1$  \\ 
        & $X^3$ & $1$ \\ \cline{3-3}   
        & & $2$  \\
         \hline
        $p=3$ & $X^3+X^2$ & $1$  \\    
        & $X^3+aX,\ a\in\mathcal C_2$ & $2$  \\
        & $X^3$ & $1$ \\ \cline{3-3}
        & & $4$ \\   \hline 
        $p>3$ & $X^3+aX,\ a\in\mathcal C_2$ & $2$  \\
        & $X^3$ & $1$ \\ \cline{3-3}
        & & $3$ \\   \hline
    \end{tabular}
\end{table}

\begin{table}[h]
\caption{Equivalence classes of $\mathcal P_{q,4}$}\label{Tb-q4}
   \renewcommand*{\arraystretch}{1.2}
    \centering
     \begin{tabular}{c|c|c}
     \hline
     	$q$ & representative & number \\ \hline      
        $q\equiv 1\pmod 6$ & $X^4+a(X^2+X),\ a\in\f_q^*$ & $q-1$  \\ 
        & $X^4+aX^2,\ a\in\mathcal C_2$ & $2$ \\ 
        & $X^4+aX,\ a\in\mathcal C_3$ & $3$ \\
        & $X^4$ & $1$ \\  \cline{3-3} 
        & & $q+5$  \\  \hline
        $q\equiv 2\pmod 6$ & $X^4+X^3+aX,\ a\in\f_q$ & $q$  \\
        & $X^4+X^2+aX,\ a\in\f_q$ & $q$  \\
        & $X^4+X$ & $1$ \\
        & $X^4$ & $1$ \\ \cline{3-3}
        & & $2q+2$ \\  \hline
        $q\equiv 3,5\pmod 6$ & $X^4+a(X^2+X),\ a\in\f_q^*$ & $q-1$  \\ 
        & $X^4+aX^2,\ a\in\mathcal C_2$ & $2$ \\ 
        & $X^4+X$ & $1$ \\
        & $X^4$ & $1$ \\  \cline{3-3} 
        & & $q+3$  \\  \hline
        $q\equiv 4\pmod 6$ &  $X^4+X^3+aX,\ a\in\f_q$ & $q$  \\
        & $X^4+X^2+aX,\ a\in\f_q$ & $q$  \\
        & $X^4+aX,\ a\in\mathcal C_3$ & $3$ \\
        & $X^4$ & $1$ \\  \cline{3-3}
        & & $2q+4$ \\  \hline 
    \end{tabular}
\end{table}

\begin{table}[h]
\caption{Equivalence classes of $\mathcal P_{q,5}$}\label{Tb-q5}
   \renewcommand*{\arraystretch}{1.2}
    \centering
     \begin{tabular}{c|c|c}
     \hline
     	$q$ & representative & number \\ \hline 
	    $q\equiv 1\pmod{12}$ & $X^5+X^4+aX^2+bX,\ a,b\in\f_q$ & $q^2$ \\
	    $p=5$ & $X^5+aX^3+bX,\ a\in\mathcal C_2,\ b\in\f_q$ & $2q$ \\
	    & $X^5+aX^2,\ a\in\mathcal C_3$ & $3$ \\
	    & $X^5+aX,\ a\in\mathcal C_4$ & $4$ \\
	    & $X^5$ & $1$ \\ \cline{3-3}
	    & & $q^2+2q+8$ \\  \hline
	    $q\equiv 1\pmod{12}$ & $X^5+a(X^3+X^2)+bX,\ a\in\f_q^*,\ b\in\f_q$ & $q^2-q$ \\
	    $p\ne 5$ & $X^5+aX^3+bX,\ a\in\mathcal C_2,\ b\in\f_q$ & $2q$ \\
	    & $X^5+a(X^2+X),\ a\in\f_q^*$ & $q-1$ \\
	    & $X^5+aX^2,\ a\in\mathcal C_3$ & $3$ \\
	    & $X^5+aX,\ a\in\mathcal C_4$ & $4$ \\
	    & $X^5$ & $1$ \\ \cline{3-3}
	    & & $q^2+2q+7$ \\  \hline 
	    $q\equiv 2,8\pmod{12}$ & $X^5+a(X^3+X^2)+bX,\ a\in\f_q^*,\ b\in\f_q$ & $q^2-q$ \\
	    & $X^5+X^3+aX,\ a\in\f_q$ & $q$ \\
	    & $X^5+X^2+aX,\ a\in\f_q$ & $q$ \\
	    & $X^5+X$ & $1$ \\
	    & $X^5$ & $1$ \\ \cline{3-3}
	    & & $q^2+q+2$ \\  \hline
	    $q\equiv 3,11\pmod{12}$ & $X^5+a(X^3+X^2)+bX,\ a\in\f_q^*,\ b\in\f_q$ & $q^2-q$ \\
	    & $X^5+aX^3+bX,\ a\in\mathcal C_2,\ b\in\f_q$ & $2q$ \\
	    & $X^5+X^2+aX,\ a\in\f_q$ & $q$ \\
	    & $X^5+aX,\ a\in\mathcal C_2$ & $2$ \\
	    & $X^5$ & $1$ \\ \cline{3-3}
	    & & $q^2+2q+3$ \\  \hline
	    $q\equiv 4\pmod{12}$ & $X^5+a(X^3+X^2)+bX,\ a\in\f_q^*,\ b\in\f_q$ & $q^2-q$ \\
	    & $X^5+X^3+aX,\ a\in\f_q$ & $q$ \\
	    & $X^5+a(X^2+X),\ a\in\f_q^*$ & $q-1$ \\
	    & $X^5+aX^2,\ a\in\mathcal C_3$ & $3$ \\
	    & $X^5+X$ & $1$ \\
	    & $X^5$ & $1$ \\ \cline{3-3}
	    & & $q^2+q+4$ \\  \hline
	\end{tabular}
\end{table}

\addtocounter{table}{-1}

\begin{table}[h]
\caption{continued}\label{Tb-q5-c}
   \renewcommand*{\arraystretch}{1.2}
    \centering
     \begin{tabular}{c|c|c}
     \hline
     	$q$ & representative & number \\ \hline 
	    $q\equiv 5\pmod{12}$ & $X^5+X^4+aX^2+bX,\ a,b\in\f_q$ & $q^2$ \\
	    $p=5$ & $X^5+aX^3+bX,\ a\in\mathcal C_2,\ b\in\f_q$ & $2q$ \\
	    & $X^5+X^2$ & $1$ \\
	    & $X^5+aX,\ a\in\mathcal C_4$ & $4$ \\
	    & $X^5$ & $1$ \\ \cline{3-3}
	    & & $q^2+2q+6$ \\  \hline
	    $q\equiv 5,9\pmod{12}$ & $X^5+a(X^3+X^2)+bX,\ a\in\f_q^*,\ b\in\f_q$ & $q^2-q$ \\
	    $p\ne 5$ & $X^5+aX^3+bX,\ a\in\mathcal C_2,\ b\in\f_q$ & $2q$ \\
	    & $X^5+X^2+aX,\ a\in\f_q$ & $q$ \\
	    & $X^5+aX,\ a\in\mathcal C_4$ & $4$ \\
	    & $X^5$ & $1$ \\ \cline{3-3}
	    & & $q^2+2q+5$ \\  \hline
	    $q\equiv 7\pmod{12}$ & $X^5+a(X^3+X^2)+bX,\ a\in\f_q^*,\ b\in\f_q$ & $q^2-q$ \\
	    & $X^5+aX^3+bX,\ a\in\mathcal C_2,\ b\in\f_q$ & $2q$ \\
	    & $X^5+a(X^2+X),\ a\in\f_q^*$ & $q-1$ \\
	    & $X^5+aX^2,\ a\in\mathcal C_3$ & $3$ \\
	    & $X^5+aX,\ a\in\mathcal C_2$ & $2$ \\
	    & $X^5$ & $1$ \\ \cline{3-3}
	    & & $q^2+2q+5$ \\  \hline
	\end{tabular}
\end{table}


\newcounter{apdx}
\newenvironment{Alem}
	{\medskip
	\refstepcounter{apdx}
	\noindent
	{\bf Lemma A\theapdx.}
	\em
	}
	{\medskip
	}
	
\newcounter{aeqno}	
	
\section*{Appendix: Counting Lemmas}

For $m,n\ge 0$, let
\begin{align*}
\alpha_{m,n}\,&=|\{(f,g):f,g\in\f_q[X]\ \text{monic},\ \deg f=m,\ \deg g=n,\ \text{gcd}(f,g)=1\}|,\cr
\alpha_{n}\,&=|\{(f,g):f,g\in\f_q[X]\ \text{monic},\ \deg f<n,\ \deg g=n,\ \text{gcd}(f,g)=1\}|.
\end{align*}

\begin{Alem}\label{LA1}
We have
\refstepcounter{aeqno}
\begin{equation}\tag{A\theaeqno}\label{A1}
\alpha_{m,n}=\begin{cases}
q^n&\text{if}\ m=0,\cr
q^{m+n}(1-q^{-1})&\text{if}\ m,n>0,
\end{cases}
\end{equation}
and
\refstepcounter{aeqno}
\begin{equation}\tag{A\theaeqno}\label{A2}
\alpha_{n}=q^{2n-1},\quad n\ge 1.
\end{equation}
\end{Alem}

\begin{proof}
For \eqref{A1}, we may assume that $n-m=d\ge 0$, and it suffices to show that
\refstepcounter{aeqno}
\begin{equation}\tag{A\theaeqno}\label{A3}
\alpha_{m,m+d}=\begin{cases}
q^d&\text{if}\ m=0,\cr
q^{2m+d}(1-q^{-1})&\text{if}\ m>0.
\end{cases}
\end{equation}
The pairs $(f,g)$, where $f,g\in\f_q[X]$ are monic, $\deg f=m$ and $\deg g=m+d$,  are of the form $(hf_1,hg_1)$, where $h,f_1,g_1\in\f_q[X]$ are monic, $\deg f_1=m-\deg h$, $\deg g_1=m+d-\deg h$, and $\text{gcd}(f_1,g_1)=1$. Hence
\[
q^{2m+d}=\sum_{i\ge 0}q^i\alpha_{m-i,m+d-i},
\]
whence
\[
\sum_{m\ge 0}q^{2m+d}X^m=\Bigl(\sum_{i\ge 0}q^iX^i\Bigr)\Bigl(\sum_{j\ge 0}\alpha_{j,j+d}X^j\Bigr).
\]
Therefore,
\begin{align*}
\sum_{j\ge 0}\alpha_{j,j+d}X^j\,&=(1-qX)\sum_{m\ge 0}q^{2m+d}X^m=q^d\Bigl(\sum_{m\ge 0}q^{2m}X^m-\sum_{m\ge 0}q^{2m+1}X^{m+1}\Bigr)\cr
&=q^d\Bigl(1+\sum_{m\ge 1}(q^{2m}-q^{2m-1})X^m\Bigr)=q^d\Bigl(1+\sum_{m\ge 1}q^{2m}(1-q^{-1})X^m\Bigr),
\end{align*}
which is \eqref{A3} (with $j$ in place of $m$).

For \eqref{A2}, we have
\begin{align*}
\alpha_n\,&=\sum_{m=0}^{n-1}\alpha_{m,n}=q^n+\sum_{m=1}^{n-1}q^{m+n}(1-q^{-1})\cr
&=q^n+q^n(q-1)\sum_{m=0}^{n-2}q^m=q^n+q^n(q^{n-1}-1)\cr
&=q^{2n-1}.
\end{align*}
\end{proof}

\begin{Alem}\label{LA2}
Let $\mathcal R_{q,n}=\{f\in\f_q[X]:\deg f=n\}$. Then
\[
|\mathcal R_{q,n}|=\begin{cases}
q-1&\text{if}\ n=0,\cr
q^{2n-1}(q^2-1)&\text{if}\ n>0.
\end{cases}
\]
\end{Alem}

\begin{proof}
For $n>0$, we have
\[
|\mathcal R_{q,n}|=(q-1)(2\alpha_n+\alpha_{n,n})=(q-1)(2q^{2n-1}+q^{2n}(1-q^{-1}))=q^{2n-1}(q^2-1).
\]
\end{proof}

For $m,n\ge 0$, let
\begin{align*}
&\beta_{m,n}=\cr
&|\{(f,g):f,g\in\f_q[X]\ \text{monic},\ \deg f=m,\ \deg g=n,\ f(0)\ne0,\ \text{gcd}(f,g)=1\}|.
\end{align*}

\begin{Alem}\label{LA3}
We have
\[
\beta_{m,n}=\begin{cases}
q^{m-n-1}(q-1)\displaystyle\frac{q^{2n+1}+1}{q+1}&\text{if}\ m>n\ge 0,\vspace{0.2em}\cr
q^n&\text{if}\ m=0,\vspace{0.2em}\cr
q^{n-m}(q-1)\displaystyle\frac{q^{2m}-1}{q+1}&\text{if}\ 1\le m\le n.
\end{cases}
\]
\end{Alem}

\begin{proof}
We have
\[
\alpha_{m,n}=\beta_{m,n}+\beta_{n,m-1}.
\]
Therefore,
\refstepcounter{aeqno}
\begin{align}\tag{A\theaeqno}\label{A4}
\beta_{m,n}\,&=\alpha_{m,n}-\beta_{n,m-1}=\alpha_{m,n}-(\alpha_{n,m-1}-\beta_{m-1,n-1})\\ 
&=\alpha_{m,n}-\alpha_{m-1,n}+\beta_{m-1,n-1}=c_{m,n}+\beta_{m-1,n-1},\nonumber
\end{align}
where
\begin{align*}
c_{m,n}\,&=\alpha_{m,n}-\alpha_{m-1,n}\cr
&=\begin{cases}
q^n&\text{if}\ m=0,\cr
q^{m-1}(q-1)&\text{if}\ m>0,\ n=0,\cr
q^{n}(q-2)&\text{if}\ m=1,\ n>0,\cr
q^{m+n-2}(q-1)^2&\text{if}\ m>1,\ n>0.
\end{cases}
\end{align*}
By \eqref{A4},
\[
\beta_{m,n}=\sum_{i\ge 0}c_{m-i,n-i}.
\]
When $m>n$,
\begin{align*}
\beta_{m,n}\,&=c_{m,n}+c_{m-1,n-1}+\cdots+c_{m-n,0}\cr
&=c_{m,n}+c_{m-1,n-1}+\cdots+c_{m-n+1,1}+q^{m-n-1}(q-1)\cr
&=\sum_{i=1}^nq^{m-n+2i-2}(q-1)^2+q^{m-n-1}(q-1)\cr
&=q^{m-n}(q-1)^2\,\frac{q^{2n}-1}{q^2-1}+q^{m-n-1}(q-1)\cr
&=q^{m-n-1}(q-1)\frac{q^{2n+1}-1}{q+1}.
\end{align*}
When $m\le n$,
\begin{align*}
\beta_{m,n}\,&=c_{m,n}+c_{m-1,n-1}+\cdots+c_{0,n-m}\cr
&=c_{m,n}+c_{m-1,n-1}+\cdots+c_{1,n-m+1}+q^{n-m}.
\end{align*}
In the above, if $m=0$,
\[
\beta_{0,n}=q^n;
\]
if $m\ge 1$,
\begin{align*}
\beta_{m,n}\,&=\sum_{i=2}^mq^{n-m+2i-2}(q-1)^2+q^{n-m+1}(q-2)+q^{n-m}\cr
&=q^{n-m+2}(q-1)^2\,\frac{q^{2(m-1)}-1}{q^2-1}+q^{n-m+1}(q-2)+q^{n-m}\cr
&=q^{n-m}(q-1)\frac{q^{2m}-1}{q+1}.
\end{align*}
\end{proof}

Let $\overline{(\ )}=(\ )^q$ be the Frobenius of $\f_{q^2}$ over $\f_q$, and for $g=\sum_{i=0}^na_iX^i\in\f_{q^2}[X]$, define $\bar g=\sum_{i=0}^n\bar a_iX^i$.
For $0\ne g\in\f_{q^2}[X]$, define $\tilde g=X^{\deg g}\bar g(X^{-1})$; that is, for $g=a_mX^m+a_{m-1}X^{m-1}+\cdots+a_0\in\f_{q^2}[X]$, $a_m\ne 0$,
\[
\tilde g=\bar a_0X^m+\bar a_1X^{m-1}+\cdots+\bar a_m.
\]
Clearly, $\widetilde{g_1g_2}=\tilde g_1\tilde g_2$, $\widetilde{X^m}=1$, and $\tilde{\tilde g}=g$ if $g(0)\ne 0$. We say the $g$ is {\em self-dual} if $\tilde g=cg$ for some $c\in\f_{q^2}^*$. In this case, $(\bar a_0,\bar a_m)=c(a_m,a_0)$, which implies that $a_0/a_m\in\mu_{q+1}$ and $c=\bar a_0/a_m\in\mu_{q+1}$.

Define
\begin{align*}\label{Gamma-def}
\Lambda_i\,&=|\{g\in\f_{q^2}[X]:\text{$g$ is monic, self-dual, $\deg g=i$}\}|,\cr
\Theta_i\,&=|\{g\in\f_{q^2}[X]:\text{$g$ is monic, $\deg g=i$, $\text{gcd}(g,\tilde g)=1$}\}|,\cr
\Gamma_{i,j}\,&=|\{(g,h):\text{$g,h\in\f_{q^2}[X]$ monic, self-dual, $\text{gcd}(g,h)=1$}\}|.
\end{align*}

\begin{Alem}\label{LA4}
We have
\refstepcounter{aeqno}
\begin{equation}\tag{A\theaeqno}\label{Lambda}
\Lambda_i=\begin{cases}
1&\text{if}\ i=0,\cr
(q+1)q^{i-1}&\text{if}\ i>0,
\end{cases} 
\end{equation}
\[
\Theta_i=\frac 1{1+q^2}\bigl[(-1)^i(1+q)+q^{2i+1}(q-1)\bigr].
\]
\end{Alem}

\begin{proof}
Every monic $g\in\f_{q^2}[X]$ has a unique representation $g=g_1h$, where $h=\text{gcd}(g,\tilde g)$, which is monic and self-dual, and $g_1\in\f_{q^2}[X]$ is monic such that $\text{gcd}(g_1,\tilde g_1)=1$. Therefore,
\[
\sum_{i=0}^l\Lambda_i\Theta_{l-i}=|\{g\in\f_{q^2}[X]\ \text{monic of degree $l$}\}|=q^{2l},
\]
that is,
\refstepcounter{aeqno}
\begin{equation}\tag{A\theaeqno}\label{A5}
\Bigl(\sum_{i=0}^\infty\Lambda_iX^i\Bigr)\Bigl(\sum_{j=0}^\infty\Theta_jX^j\Bigr)=\sum_{l=0}^\infty q^{2l}X^l=\frac 1{1-q^2X}.
\end{equation}
Clearly $\Lambda_0=1$. Assume $l\ge 1$. Let $g(X)=X^l+a_{l-1}X^{l-1}+\cdots+a_0\in\f_{q^2}[X]$, so $\tilde g(X)=\bar a_0X^l+\bar a_1X^{l-1}+\cdots+1$. Then $g$ is self-dual if and only if
\refstepcounter{aeqno}
\begin{equation}\tag{A\theaeqno}\label{A6}
\begin{array}{cccccccl}
&\bar a_0&\kern-0.6em(\kern-0.6em&a_0& a_1& \dots& a_{l-1}&\kern-0.6em)\vspace{0.2em}\cr
=&&\kern-0.6em(\kern-0.6em&1& \bar a_{l-1}& \dots& \bar a_1&\kern-0.6em).
\end{array}
\end{equation}
If $l-1$ is even, to satisfy
\[
\begin{array}{ccccccccccl}
&\bar a_0&\kern-0.6em(\kern-0.6em&a_0& a_1&\dots&a_{(l-1)/2}&a_{(l+1)/2}& \dots& a_{l-1}&\kern-0.6em)\vspace{0.2em}\cr
=&&\kern-0.6em(\kern-0.6em&1& \bar a_{l-1}&\dots&\overline{a_{(l+1)/2}}&\overline{a_{(l-1)/2}}& \dots& \bar a_1&\kern-0.6em),
\end{array}
\]
we can choose $a_0\in\mu_{q+1}$, choose $a_1,\dots,a_{(l-1)/2}\in\f_{q^2}$ arbitrarily and let $a_i=\overline{a_{l-i}}/\bar a_0$ for $(l+1)/2\le i\le l-1$. Hence $\Lambda_l=(q+1)(q^2)^{(l-1)/2}=(q+1)q^{l-1}$. If $l-1$ is odd, to satisfy
\[
\begin{array}{cccccccccccl}
&\bar a_0&\kern-0.6em(\kern-0.6em&a_0& a_1&\dots&a_{l/2-1}&a_{l/2}&a_{l/2+1}& \dots& a_{l-1}&\kern-0.6em)\vspace{0.2em}\cr
=&&\kern-0.6em(\kern-0.6em&1& \bar a_{l-1}&\dots&\overline{a_{l/2+1}}&\overline{a_{l/2}}&\overline{a_{l/2-1}}& \dots& \bar a_1&\kern-0.6em),
\end{array}
\]
we can choose $a_0\in\mu_{q+1}$, choose $a_1,\dots,a_{l/2-1}\in\f_{q^2}$ arbitrarily, choose $a_{l/2}\in\f_{q^2}$ such that $\bar a_0 a_{l/2}=\overline{a_{l/2}}$ and let $a_i=\overline{a_{l-i}}/\bar a_0$ for $\l/2+1\le i\le l-1$. Since $a_0\in\mu_{q+1}$, the number of choices for $a_{l/2}$ is $q$. Thus we also have $\Lambda_l=(q+1)q(q^2)^{l/2-1}=(q+1)q^{l-1}$. Therefore,
\[
\Lambda_l=\begin{cases}
1&\text{if}\ l=0,\cr
(q+1)q^{l-1}&\text{if}\ l>0.
\end{cases}
\]
We then have
\refstepcounter{aeqno}
\begin{align}\tag{A\theaeqno}\label{A7}
\sum_{i=0}^\infty\Lambda_iX^i\,&=1+\sum_{i=1}^\infty(q+1)q^{i-1}X^i=\sum_{i=0}^\infty(q+1)q^{i-1}X^i+1-(q+1)q^{-1}\\
&=\frac{q+1}q\frac 1{1-qX}-\frac 1q=\frac{1+X}{1-qX}.\nonumber
\end{align}
By \eqref{A5} and \eqref{A7}, 
\begin{align*}
\sum_{j=0}^\infty\Theta_jX^j\,&=\frac{1-qX}{1+X}\cdot\frac 1{1-q^2X}=\frac{1+q}{1+q^2}\frac 1{1+X}+\frac{q(q-1)}{1+q^2}\frac 1{1-q^2X}\cr
&=\frac{1+q}{1+q^2}\sum_{j=0}^\infty(-1)^jX^j+\frac{q(q-1)}{1+q^2}\sum_{j=0}^\infty q^{2j}X^j\cr
&=\frac 1{1+q^2}\sum_{j=0}^\infty\bigl[(-1)^j(1+q)+q^{2j+1}(q-1)\bigr]X^j.
\end{align*}
Hence
\[
\Theta_j=\frac 1{1+q^2}\bigl[(-1)^j(1+q)+q^{2j+1}(q-1)\bigr].
\]
\end{proof}

\begin{Alem}\label{LA5}
For $i,j\ge 0$, we have
\[
\Gamma_{i,i+j}=\begin{cases}
1&\text{if}\ i=j=0,\cr
q^{j-1}(q+1)&\text{if}\ i=0,\ j>0,\vspace{0.3em}\cr
\displaystyle\frac{q(q+1)}{q^2+1}(q^{2i}-q^{2i-2}-(-1)^i2)&\text{if}\ i>0,\ j=0,\vspace{0.3em}\cr
\displaystyle\frac{q^{j-1}(q+1)(q^2-1)}{q^2+1}(q^{2i}-(-1)^i)&\text{if}\ i>0,\ j>0.
\end{cases}
\]
\end{Alem}

\begin{proof}
Each ordered pair $(f,g)$, where $f,g\in\f_{q^2}[X]$ are monic and self-dual with $\deg f=i$ and $\deg g=i+j$, has a unique representation $(f,g)=(f_1h,g_1h)$, where $f_1,g_1,h\in\f_{q^2}[X]$ are monic and self-dual and $\text{gcd}(f_1,g_1)=1$. Thus
\[
\sum_k\Lambda_k\Gamma_{i-k,i+j-k}=\Lambda_i\Lambda_{i+j}.
\]
Therefore,
\refstepcounter{aeqno}
\begin{equation}\tag{A\theaeqno}\label{A8}
\Bigl(\sum_{k\ge 0}\Lambda_kX^k\Bigr)\Bigl(\sum_{l\ge 0}\Gamma_{l,l+j}X^l\Bigr)=\sum_{i\ge 0}\Lambda_i\Lambda_{i+j}X^i.
\end{equation}
When $j=0$, by \eqref{Lambda},
\refstepcounter{aeqno}
\begin{align}\tag{A\theaeqno}\label{A9}
\sum_{i\ge 0}\Lambda_i\Lambda_iX^i\,&=1+\sum_{i\ge 1}(q+1)^2q^{2(i-1)}X^i\\
&=\sum_{i\ge 0}(q+1)^2q^{2(i-1)}X^i+1-(q+1)^2q^{-2}\cr
&=(q+1)^2q^{-2}\frac 1{1-q^2X}+1-(q+1)^2q^{-2}\cr
&=\frac{1+(2q+1)X}{1-q^2X}.\nonumber
\end{align}
Combining \eqref{A8}, \eqref{A7} and \eqref{A9} gives
\begin{align*}
\sum_{l\ge 0}\Gamma_{l,l}X^l\,&=\frac{1-qX}{1+X}\cdot\frac{1+(2q+1)X}{1-q^2X}\cr
&=\frac{2q+1}q-\frac{2q(q+1)}{q^2+1}\cdot\frac 1{1+X}+\frac{(q-1)(q+1)^2}{q(q^2+1)}\cdot\frac 1{1-q^2X}\cr
&=\frac{2q+1}q-\frac{2q(q+1)}{q^2+1}\sum_{l\ge 0}(-1)^lX^l+\frac{(q-1)(q+1)^2}{q(q^2+1)}\sum_{l\ge 0}q^{2l}X^l.
\end{align*}
Hence
\[
\Gamma_{l,l}=\begin{cases}
1&\text{if}\ l=0,\vspace{0.2em}\cr
\displaystyle\frac{q(q+1)}{q^2+1}(q^{2l}-q^{2l-2}-(-1)^l2)&\text{if}\ l>0.
\end{cases}
\]
When $j>0$, by \eqref{Lambda},
\refstepcounter{aeqno}
\begin{align}\tag{A\theaeqno}\label{A10}
\sum_{i\ge 0}\Lambda_i\Lambda_{i+j}X^i\,&=(q+1)q^{j-1}+\sum_{i\ge 1}(q+1)^2q^{2i+j-2}X^i\\
&=\sum_{i\ge 0}(q+1)^2q^{2i+j-2}X^i+(q+1)q^{j-1}-(q+1)^2q^{j-2}\cr
&=(q+1)^2q^{j-2}\frac 1{1-q^2X}-(q+1)q^{j-2}\cr
&=q^{j-1}(q+1)\frac{1+qX}{1-q^2X}.\nonumber
\end{align}
Combining \eqref{A8}, \eqref{A7} and \eqref{A10} gives
\begin{align*}
\sum_{l\ge 0}\Gamma_{l,l+j}X^l\,&=q^{j-1}(q+1)\frac{1-qX}{1+X}\cdot\frac{1+qX}{1-q^2X}\cr
&=q^{j-1}(q+1)\Bigl(1+\frac{1-q^2}{1+q^2}\cdot\frac 1{1+X}-\frac{1-q^2}{1+q^2}\cdot\frac1{1-q^2X}\Bigr)\cr
&=q^{j-1}(q+1)+\frac{q^{j-1}(q+1)(1-q^2)}{1+q^2}\Bigl(\sum_{l\ge 0}(-1)^lX^l-\sum_{l\ge 0}q^{2l}X^l\Bigr).
\end{align*}
Hence
\[
\Gamma_{l,l+j}=\begin{cases}
q^{j-1}(q+1)&\text{if}\ l=0,\vspace{0.3em}\cr
\displaystyle\frac{q^{j-1}(q+1)(q^2-1)}{q^2+1}(q^{2l}-(-1)^l)&\text{if}\ l>0.
\end{cases}
\].
\end{proof}



\end{document}